\documentclass[11pt]{amsart}
\usepackage[left=1.25in,right=1.25in]{geometry}
\usepackage{amsthm}
\usepackage{amsmath}
\usepackage{mathtools}
\usepackage{amssymb}
\usepackage{hyperref}
\usepackage{tikz-cd}
\usepackage{tikz}
\usepackage{array}
\usetikzlibrary{decorations.markings}
\usepackage{enumitem}
\usepackage{amsfonts}
\usepackage{hyperref}
\usepackage{todonotes}
\usepackage[utf8]{inputenc}
\usepackage{ytableau}
\usepackage{subcaption}
\usepackage[export]{adjustbox}
\usepackage{mathdots}
\usepackage{comment}

\newtheorem{theorem}{Theorem}[section]
\newtheorem{prop}[theorem]{Proposition}

\newtheorem{lemma}[theorem]{Lemma}
\newtheorem{cor}[theorem]{Corollary}

\theoremstyle{definition}
\newtheorem{ex}[theorem]{Example}

\newtheorem{defin}[theorem]{Definition}
\newtheorem{remark}[theorem]{Remark}

\makeatletter
\newcommand\circled[1]{%
  \mathpalette\@circled{#1}%
}
\newcommand\@circled[2]{%
  \tikz[baseline=(math.base)] \node[draw,circle,inner sep=1pt] (math) {$\m@th#1#2$};%
}
\makeatother

\DeclareMathOperator{\Red}{Red}
\DeclareMathOperator{\Inv}{Inv}
\DeclareMathOperator{\SYT}{SYT}
\DeclareMathOperator{\SDT}{SDT}
\DeclareMathOperator{\BS}{BS}
\DeclareMathOperator{\rk}{rk}
\DeclareMathOperator{\ro}{ro}
\DeclareMathOperator{\sh}{sh}

\title{Balanced Shifted Tableaux}

\author{Jiyang Gao}
\address{Department of Mathematics, Harvard University, Cambridge, MA 02138}
\email{\href{mailto:jgao@math.harvard.edu}{{\tt jgao@math.harvard.edu}}}

\author{Shiliang Gao}
\address{Department of Mathematics, University of Illinois at Urbana-Champaign, Urbana, IL 61801}
\email{\href{mailto:sgao23@illinois.edu}{{\tt sgao23@illinois.edu}}}

\author{Yibo Gao}
\address{Department of Mathematics, Massachusetts Institute of Technology, Cambridge, MA 02139}
\email{\href{mailto:gaoyibo@mit.edu}{{\tt gaoyibo@mit.edu}}}

\date{\today}

\begin{document}
\begin{abstract}
We introduce balanced shifted tableaux, as an analogue of balanced tableaux of Edelman and Greene, from the perspective of root systems of type $B$ and $C$. We show that they are equinumerous to standard Young tableaux of the corresponding shifted shape by presenting an explicit bijection. 
\end{abstract}
\maketitle

\section{Introduction}\label{sec:intro}
In their seminal paper \cite{BalancedTableaux}, Edelman and Greene introduced \emph{balanced tableaux} and showed that they are equinumerous to standard Young tableaux of the same shape. They defined the \emph{Edelman-Greene insertion} which yields a bijective proof of the reduced words of the longest permutation being equinumerous to standard Young tableaux of staircase shape, a result due originally to Stanley \cite{stanley1984number}. Fomin, Greene, Reiner and Shimozono \cite{fomin1997balanced} later generalized this enumeration result to diagrams and related the story to Schubert polynomials. 

Shifted tableaux, just as Young tableaux, are also algebraically and combinatorially meaningful (see for example \cite{sagan1987shifted,WorleyThesis}). In this paper, we define \emph{balanced shifted tableaux} (Definition~\ref{def:balanced}), as an analogue to balanced tableaux,
%using certain rank conditions on the hooks, 
from the perspective of root systems of type $B$ and $C$. The following is our main theorem, which says that balanced shifted tableaux are equinumerous to standard shifted tableaux. 
\begin{theorem}\label{thm:main}
For a shifted shape $\lambda$, the number of standard Young tableaux of shape $\lambda$ equals the number of balanced shifted tableaux of shape $\lambda$. 
\end{theorem}
We prove Theorem~\ref{thm:main} by presenting an explicit bijection between the two sets of objects, $\SYT(\lambda)$ and $\BS(\lambda)$. 
%See Section~\ref{sec:prelim} for detailed notations.
Specifically, we have the following chain of bijections: 
\[\SYT(\lambda)\longleftrightarrow\SYT(Z(d,r))|_{\lambda}\longleftrightarrow\Red(w^{\lambda})\longleftrightarrow\BS(Z(d,r))|_{\lambda}\longleftrightarrow\BS(\lambda),\]
where we address each step separately. We defer the definition of $\SYT(\lambda)$ and $\BS(\lambda)$ to Section~\ref{sec:prelim} and the definition of $\SYT(Z(d,r))|_{\lambda}$, $w^{\lambda}$ and $\BS(Z(d,r))|_{\lambda}$ to Section~\ref{sec:general}.
Here, $\SYT(\lambda)\rightarrow \SYT(Z(d,r))|_{\lambda}$ and $\BS(\lambda)\rightarrow\BS(Z(d,r))|_{\lambda}$ are the procedures to pad a tableaux from shape $\lambda$ to a large trapezoid $Z(d,r)$, while the middle steps utilize type $B$ Edelman-Greene insertion defined by Kra\'skiewicz \cite{K89}.
Our strategy largely follows the framework of Edelman and Greene \cite{BalancedTableaux}, with the main difference that double staircases, which are the analogues of staircases in type $B$, are no longer sufficient for padding purposes.

The remainder of the paper is organized as follows. In Section~\ref{sec:prelim}, we introduce key definitions and necessary background. In Section~\ref{sec:trapezoid}, we discuss the trapezoid shape $Z(d,r)$ which serves the purpose of the staircase shape in type $A$, and provide a bijection between $\BS(Z(d,r))$ and $\Red(w^{(d,r)})$, where $w^{(d,r)}$ is certain signed permutation. In Section~\ref{sec:kraskiewicz}, we provide a bijection between $\SYT(Z(d,r))$ and $\Red(w^{(d,r)})$ via Kra\'skiewicz's insertion algorithm, thus establishing the main theorem for the trapezoid case. Finally in Section~\ref{sec:general}, we finish the proof of Theorem~\ref{thm:main} for general shifted shapes by restricting the above bijections to $\lambda\subset Z(d,r)$.
\section{Definitions and Preliminaries}\label{sec:prelim}
\subsection{Strict partitions and shifted tableaux}
A \textit{strict partition} $\lambda$ is a sequence of strictly decreasing positive integers $(\lambda_1>\lambda_2>\cdots>\lambda_d>0)$, where $d$ is the number of (nonzero) parts of $\lambda$. We denote $|\lambda| = \sum_{i = 1}^{d} \lambda_i$ as the \textit{size} of $\lambda$. For a strict partition $\lambda$ its corresponding %\textit{shifted tableaux}, or
\textit{shifted shape}, consists of $\lambda_i$ boxes in row $i$, shifted $d-i+1$ steps to the left. More specifically, the shifted shape is the diagram
\[D(\lambda):=\{(i,j{-}d{+}i{-}1)\:|\: 1\leq i\leq d,\ 1\leq j\leq \lambda_i\}.\]
For simplicity of notation, we also use $\lambda$ to denote its shape $D(\lambda)$. Note that for a shifted shape, its columns $-(d-1),\ldots,0$ form a staircase shape of length $d$ flipped horizontally. For a shifted shape $\lambda$, define a \emph{shifted tableaux} $T$ to be a filling of $D(\lambda)$ with non-negative integers. For any shifted tableaux $T$, let $\sh(T)$ denote its underlying shifted shape. 

Throughout the paper, we fix the number $d$, that is the length of all the shifted shapes we are going to consider. We also write $\bar i$ to mean $-i$. 

\begin{defin}\label{def:standard}
A shifted tableaux $T$ of shape $\lambda$ is called a \emph{standard Young tableaux} if it is a filling of $1,2,\ldots,|\lambda|$ that is increasing in rows and columns.
%A \textit{standard Young tableaux} of shifted shape $\lambda$ is a filling of $\lambda$ with $1,2,\ldots,|\lambda|$, increasing in rows and columns.
\end{defin}
The set of standard Young tableaux of shape $\lambda$ is denote $\SYT(\lambda)$ and its cardinality is denoted $f^{\lambda}$. The number $f^{\lambda}$ can be computed via the hook length formula as we explain here. For a box $(i,j)\in\lambda$ with $j\geq0$, its \textit{hook} $H(i,j)$ consists of all the boxes in row $i$ to the right of $(i,j)$, all the boxes in column $j$ below $(i,j)$ and the box $(i,j)$ itself. For a box $(i,\bar j)\in\lambda$ with $j>0$, its \textit{hook} $H(i,\bar j)$ consists of all the boxes in row $i$ to the right of $(i,\bar j)$, all the boxes in column $\bar j$ below $(i,\bar j)$, the box $(i,\bar j)$ itself and all the boxes in row $d-j+1$. Let $h(i,j)=|H(i,j)|$ be the size of the hook.
\begin{theorem}\cite{Thrall}
For a shifted shape $\lambda$, $f^{\lambda}=|\lambda|!\,/\prod_{x\in\lambda}h(x).$
\end{theorem}

To define an analogous notion of balanced tableaux, as in \cite{BalancedTableaux}, for shifted shapes, we need some more notions. For a filling $B$ of shape $\lambda$, its \textit{extended filling} $\tilde{B}$ is a filling of the extended shape \[\tilde{\lambda}=\lambda\cup\{(1,\bar d),(2,\overline{d{-}1}),\ldots(d,\bar1)\}\] which agrees with $B$ on $\lambda$ and equals $B(i,0)$ on the newly added box $(i,-(d+1-i))$. The \textit{extended hook} is defined as $\tilde{H}(i,j)=H(i,j)$ for $j\geq0$, and $\tilde{H}(i,\bar j)=H(i,j)\cup\{(d+1-j,\bar j)\}$ for $j>0$. See Example~\ref{ex:BST621} for visualization.

For a box $(i,j)\in\lambda$, we also define its \textit{rank function} $\rk(i,j)$. If $j\geq0$, let $\rk(i,j)$ be the number of boxes in row $i$ of $H(i,j)$, and let $\rk(i,\bar j)$ be $2$ plus the number of boxes in $H(i,j)$ with positive column index. More formally,
\[
\rk(i,j)=\begin{cases}
\lambda_i-d+i-j\ &\text{if }j\geq0,\\
\lambda_i-d+i+\lambda_{d+1+j}+j+1 &\text{if }j<0.
\end{cases}
\]
We can now introduce our main object of study:
\begin{defin}\label{def:balanced}
A shifted tableaux $B$ of shape $\lambda$ is called a \textit{balanced shifted tableaux} if it is a filling of $1,2,\ldots,|\lambda|$ such that $B(i,j)$ is the $\rk(i,j)$-th largest entry in the extended hook $\tilde H(i,j)$ of $\tilde{B}$ for all $(i,j)\in \lambda$. Define $\BS(\lambda)$ to be the set of balanced shifted tableaux of shape $\lambda$.
\end{defin}
%For a shifted shape $\lambda$, let $\BS(\lambda)$ be the set of balanced shifted tableaux $B$ of shape $\lambda$, and let $b^{\lambda}$ be its cardinality. 

\begin{remark}
If we instead naively define $\rk(i,j)$ to be the length of the \textit{right arm} of $H(i,j)$ as in straight shapes, i.e. define the balanced condition to be $B(i,j)$ remains unchanged after reordering the elements in $H(i,j)$ (or $\tilde{H}(i,j)$ ), then the number of such tableaux is different from $f^{\lambda}$.
\end{remark}

\begin{ex}\label{ex:BST621}
Let $\lambda=(6,2,1)$ and consider the balanced shifted tableaux in Figure~\ref{fig:BST621}. The hook $H(1,-1)$ contains the colored boxes so $h(1,-1)=7$, while the extended hook $\tilde{H}(1,-1)$ contains one more box at coordiante $(3,-1)$, which is circled and filled with $1$. As this hook contains $3$ boxes with positive column index, we have $\rk(1,-1)=5$. The balanced condition is now satisfied at coordinate $(1,-1)$ as $3$ is indeed the $5$-th largest numbers among the numbers in the extend hook, $9,5,2,4,3,7,1,1$.
\begin{figure}[h!]
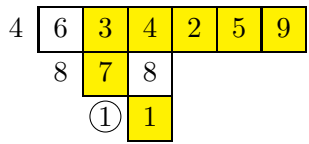

\centering
\ytableausetup{boxsize=1.5em}
\begin{ytableau}
\none[4]&6&*(yellow)3&*(yellow)4&*(yellow)2&*(yellow)5&*(yellow)9\\
\none&\none[8]&*(yellow)7&8&\none&\none&\none\\
\none&\none&\none[\circled{1}]&*(yellow)1&\none&\none&\none
\end{ytableau}
\caption{A balanced shifted tableau of shape $(6,2,1)$}
\label{fig:BST621}
\end{figure}
\end{ex}
\begin{remark}
Here is another way to understand the extended hooks $\tilde{H}(i,j)$ and ranks $\rk(i,j)$, shown in Figure~\ref{fig:BSTsymmetric}. Given a shifted tableau $B$, we stack its extended filling $\tilde{B}$ and a flipped copy $B^T$ (see left of Figure~\ref{fig:BSTsymmetric}) together to obtain a larger tableau $B_0$. The entries in the shaded boxes on the diagonal agree with column $0$ of $B$. Then the extended hook $\tilde{H}(i,j)$ of $x=(i,j)$ is the same as the standard hook of $x$ in $\tilde{B}$ (colored in blue). The rank $\rk(i,j)$ is the number of yellow boxes (see right of Figure~\ref{fig:BSTsymmetric}) in the hook of $x$.
\begin{figure}[h!]
\centering
\begin{tikzpicture}[scale = 0.4]
\fill[yellow!100] (-4,4) rectangle (-3,3);
\fill[yellow!100] (-3,3) rectangle (-2,2);
\fill[yellow!100] (-2,2) rectangle (-1,1);
\fill[yellow!100] (-1,1) rectangle (0,0);
% \draw (0,0) -- ++(2,0) --++(0,1) -- ++(1,0) --++(0,1) --++(2,0) --++(0,2) --++(-8,0) --++(0,-1) --++(1,0) --++(0,-1) --++(1,0) --++(0,-1) --++(1,0) -- cycle;
\draw (0,0) -- ++(2,0) --++(0,1) -- ++(1,0) --++(0,1) --++(2,0) --++(0,2) --++(-9,0) --++(0,-1);
\node at (1,2) {$\tilde{B}$};
\draw (0,0) -- ++(0,-2) --++(-1,0) -- ++(0,-1) --++(-1,0) --++(0,-2) --++(-2,0) --++(0,8) --++(1,0) --++(0,-1) --++(1,0) --++(0,-1) --++(1,0) --++(0,-1) -- cycle;
\node at (-2,-1) {$B^T$};
% \draw (-3,4) -- (-4,4) -- (-4,3);

\begin{scope}[xshift=400pt]
\fill[yellow] (0,3) rectangle (5,2);
\fill[yellow] (-2,0) rectangle (-1,-3);
\draw (0,0) -- ++(2,0) --++(0,1) -- ++(1,0) --++(0,1) --++(2,0) --++(0,2) --++(-9,0) --++(0,-9) --++(2,0) --++(0,2) --++(1,0) --++(0,1) --++(1,0) -- cycle;
\draw[dashed] (-4,0) -- (0,0) -- (0,4);
\draw[blue] (-2,3) rectangle (-1,2);
\node[blue] at (-1.5,2.5) {$x$};
\draw[blue] (-1,3) rectangle (5,2);
\draw[blue] (-2,2) rectangle (-1,-3);
\end{scope}
\end{tikzpicture}
\caption{An alternative description of $\rk(i,j)$}
\label{fig:BSTsymmetric}
\end{figure}
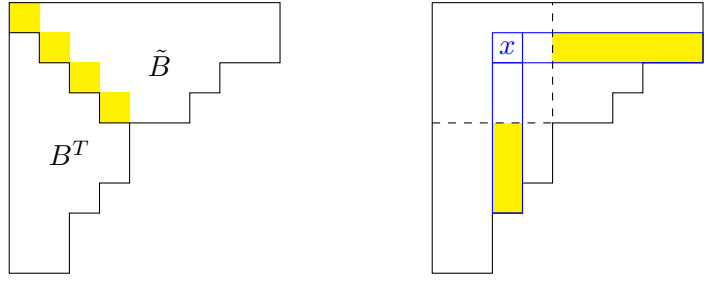
\end{remark}

\subsection{Root systems and Weyl groups}
Readers are referred to \cite{humphreys} for detailed exposition on root systems and Weyl groups. Let $\Phi\subset V\simeq\mathbb{R}^d$ be a finite crystallographic root system of rank $d$, with a chosen set of positive roots $\Phi^+$ which corresponds to a set of simple roots $\Delta=\{\alpha_0,\alpha_1,\ldots,\alpha_{d-1}\}$. Let $s_{\alpha}$ be the reflection across the hyperplane normal to $\alpha$, and write $s_i$ for the simple reflections $s_{\alpha_i}$. Let $W(\Phi)\subset\mathrm{GL}(V)$ be the finite Weyl group, defined to be generated by $s_0,\ldots,s_{d-1}$. 

For $w\in W(\Phi)$, let $\ell(w)$ denote its Coxeter length, which equals the size of its (left) inversion set $\Inv(w):=\Phi^+\cap w\Phi^-$. For any sequence $\mathbf{a} = (a_1,a_2,\ldots,a_{\ell(w)})$, we say $\mathbf{a}$ is a reduced word of $w$ if $w = s_{a_1}s_{a_2},\ldots,s_{a_{\ell(w)}}$. 
Let $\Red(w)$ be the set of reduced words of $w$. For each reduced word $\mathbf{a}\in \Red(w)$,
%$R=s_{i_1}\cdots s_{i_{\ell}}\in\Red(w)$, 
%% Shiliang: I changed the definition of Red(w) to the sequence of indices so that it's easier to state the definition of insertion.
its \textit{(total) reflection order} is an ordering $\ro(\mathbf{a})=\gamma_1,\ldots,\gamma_{\ell(w)}$ of $\Inv(w)$ where $\gamma_j=s_{a_1}\cdots s_{a_{j-1}}\alpha_j\in\Phi^+$. 
Let 
\[\ro(w) = \{\ro(\mathbf{a}):\mathbf{a}\in \Red(w)\}.\]
The following proposition is classical and very useful, which follows immediately from the biconvexity classification of inversion sets. See for example Proposition 3 of \cite{bjorner1984orderings}.
\begin{prop}\label{prop:root-ordering}
Let $\gamma = \gamma_1,\ldots,\gamma_{\ell(w)}$ be an ordering of $\Inv(w)$. Then $\gamma\in \ro(w)$
%A sequence of roots $\gamma_1,\ldots,\gamma_\ell$ of $\Inv(w)$ is a reflection order of some reduced word of $w$ 
if and only if for all the triples $\alpha,\beta,\alpha+\beta\in\Phi^+$ such that $\alpha,\alpha+\beta\in\Inv(w)$,
\begin{enumerate}
\item if $\beta\notin\Inv(w)$, then $\alpha$ appears before $\alpha+\beta$ in this sequence;
\item and if $\beta\in\Inv(w)$, then $\alpha+\beta$ appears in the middle of $\alpha$ and $\beta$. 
\end{enumerate}
\end{prop}

We are primarily concerned with root systems of type $B_n$, and adopt the following convention, where $e_i$ is the $i$-th coordinate vector:
\begin{itemize}
\item $\Phi(B_n)=\{\pm e_j\pm e_i\:|\: 1\leq i<j\leq n\}\cup\{\pm e_i\:|\: 1\leq i\leq n\}$;
\item $\Phi^+(B_n)=\{e_j\pm e_i\:|\: 1\leq i<j\leq n\}\cup\{e_i\:|\: 1\leq i\leq n\}$;
\item $\Delta=\{\alpha_0=e_1,\alpha_1=e_2-e_1,\ldots,\alpha_{n-1}=e_{n}-e_{n-1}\}$;
\item $W(B_n)=\{\text{permutation }w\text{ on }1,\ldots,n,\bar 1,\ldots,\bar n\:|\: w(i)=-w(\bar i),\ \forall i\}.$
\end{itemize}
The type $B_n$ Weyl group $W(B_n)$ is called the group of \textit{signed permutations}. For a signed permutation $w$, its one-line notation is written as $w(1)w(2)\cdots w(n)$. For example, $w=3\bar 42\bar 1\in W(B_4)$ means that $w(1)=3$, $w(2)=-4$, $w(3)=2$ so that $w(-3)=-2$ and $w(4)=-1$. A reduced word of $w\in W(B_n)$ can be viewed as going from $\mathrm{id}=12\cdots n$ to $w$ by swapping adjacent entries (and their negatives) one step at a time, while the corresponding reflection order records $e_j-e_i$ if the values $j$ and $i$ are swapped (and records $e_i$ if $i$ and $\bar i$ are swapped).
\begin{ex}
Consider $w=1\bar342\in W(B_4)$ with a reduced word 
$\mathbf{a} = 21031\in \Red(w)$.
%$R=s_2s_1s_0s_3s_1\in\Red(w)$
We compute its reflection order to be $e_3-e_2,e_3-e_1,e_3,e_4-e_2,e_3+e_1$, which can be seen as follows:
\[
\begin{tikzcd}
1234 \arrow[r, "e_3-e_2"]& 1324 \arrow[r, "e_3-e_1"] & 3124 \arrow[r, "e_3"] & \bar3124 \arrow[r, "e_4-e_2"] & \bar3142 \arrow[r, "e_3+e_1"] & 1\bar342.
\end{tikzcd}
\]
\end{ex}
\section{Bijection between \texorpdfstring{$\BS(Z(d,r))$}{} and \texorpdfstring{$\Red(w^{(d,r)})$}{} via reflection order}\label{sec:trapezoid}
%\section{The trapezoid \texorpdfstring{$Z(d,r)$}{} and its signed permutation \texorpdfstring{$w^{(d,r)}$}{}}\label{sec:trapezoid}
A crucial shape for our analysis is the \emph{trapezoid} 
\[Z(d,r):=(r+2d-1,r+2d-3,\ldots,r+3,r+1)\]
with height $d$ and base lengths $r+2d-1$ and $r+1$. In particular, $Z(d,0)$ is the \textit{double staircase} and every shifted shape is contained in some trapezoid of the same height. 

Set $e_{-j} = -e_j$ for all $j>0$ and $e_0=0$, and consider the labeling $f:Z(d,r) \longrightarrow \Phi^{+}(B_{d+r})$ where
\begin{equation}\label{eqn:defrootlabel}
    \begin{split}
        f(i,j) = 
\begin{cases}
    e_{d+1-i}-e_{j} & \text{if } j\leq 0,\\
    % e_i &  \text{if } j = 0,\\
    e_{d+1-i}+e_{j+d}&\text{if }0<j\leq r,\\
    e_{d+1-i}-e_{j-r}&\text{if }j>r.
\end{cases}
    \end{split}
\end{equation}
Define the permutation $w^{(d,r)}\in W(B_{d+r})$ associated to $Z(d,r)$ by
\begin{equation}\label{eqn:defw}
    \begin{split}
        w^{(d,r)}(i) := 
\begin{cases}
    d+i &\text{if }0<i\leq r,\\
    \overline{i-r} &\text{if }i>r.
\end{cases}
    \end{split}
\end{equation}
\begin{prop}
For all $d>0$ and $r\geq 0$, $f(Z(d,r)) = \Inv(w^{(d,r)})$.
\end{prop}
\begin{proof}
    Since for $i>0$ and $j<i$ such that $j\neq \bar i$, $e_i-e_j\in \Inv(w)$ if and only if $w^{-1}(i)<w^{-1}(j)$. By \eqref{eqn:defw}, we have
    \begin{equation}\label{eqn:Mar6aaa}
        e_i-e_j\in \Inv(w^{(d,r)})\iff i\in [d+r],|j|\leq d, i>j\neq -i.
    \end{equation}
    Since $e_i\in \Inv(w)\iff w^{-1}(i)<0$ for all $i>0$, we get 
    \begin{equation}\label{eqn:Mar6bbb}
        e_i\in \Inv(w^{(d,r)})\iff i\in [d].
    \end{equation} 
    We are then done by comparing \eqref{eqn:Mar6aaa} and \eqref{eqn:Mar6bbb} with \eqref{eqn:defrootlabel}.
\end{proof}

The labeling $f$ can also be extended to a labeling $\tilde{f}:\tilde{Z}(d,r)\to \Phi^+(B_{d+r})$ where
\[\tilde{Z}(d,r)=Z(d,r)\cup\{(1,\bar d),(2,\overline{d-1}),\ldots(d,\bar1)\}\]
is the extended shape of $Z(d,r)$ with $d$ extra boxes as defined in Section~\ref{sec:prelim}. The extended labeling is given by
\[\tilde{f}(i,j)=\begin{cases}
    2e_{d+1-i} &\text{if }j=\overline{d+1-i},\\
    f(i,j) &\text{otherwise}.
\end{cases}\]

\begin{ex}\label{ex:45123}
For $d = 3$ and $r = 2$, we have $Z(d,r) = (7,5,3)$ and $w^{(3,2)} = 45\bar{1}\bar{2}\bar{3}\in W(B_5)$. See Figure~\ref{fig:d=3,r=2} for the extended labeling $\tilde f$ in this case.
\begin{figure}[h!]
\centering
\ytableausetup{boxsize=3em}
\begin{ytableau}
\none[2e_3] & e_3{+}e_2 & e_3{+}e_1 & e_3 & e_4{+}e_3 & e_5{+}e_3 & e_3{-}e_1 & e_3{-}e_2 \\
\none & \none[2e_2] & e_2{+}e_1 & e_2 & e_4{+}e_2 & e_5{+}e_2 & e_2{-}e_1\\
\none & \none & \none[2e_1] & e_1 & e_4{+}e_1 & e_5{+}e_1
\end{ytableau}
\caption{The extended labeling $\tilde{f}$ of $\tilde{Z}(3,2)$}
\label{fig:d=3,r=2}
\end{figure}
\end{ex}
The filling of a balanced shifted tableaux $B\in\BS(Z(d,r))$ can be viewed as a map $B:Z(d,r)\to\mathbb{N}$ by sending a box to its entry. Then the composition $Bf^{-1}:\Inv(w^{(d,r)})\to \mathbb{N}$ encodes an ordering of the roots in $\Inv(w^{(d,r)})$. We will show that this actually gives a reflection order in $\ro(w^{(d,r)})$.
\begin{prop}\label{prop:root-order-balanced}
The map $B\mapsto Bf^{-1}$ is a bijection between $\BS(Z(d,r))$ and $\ro(w^{(d,r)})$, and thus induces a bijection between $\BS(Z(d,r))$ and $\Red(w^{(d,r)})$.
\end{prop}
To prove Proposition~\ref{prop:root-order-balanced}, we need the following technical lemmas.

\begin{lemma}\label{lm:same-col-order}
    Given $B\in \BS(\lambda)$ and $j\geq 0$, if column $j$ and $j+1$ have the same length in $B$, then $B(i,j)<B(i,j+1)$ for all $i$.
\end{lemma}
\begin{proof}
    We induce on the the number of boxes directly below box $(i,j)$. If there are no boxes below $(i,j)$, then $B(i,j)$ is the minimum in its hook $H(i,j)$ and $B(i,j)<B(i,j+1)$. Now suppose $B(i,j)<B(i,j+1)$ for all $i>k$. Since $B$ is balanced and $j\geq 0$, $B(k,j)$ is smaller than $\rk(k,j)-1$ entries in its hook $H(k,j)$. If $B(k,j)>B(k,j+1)$, since $B(i,j)<B(i,j+1)$ for all $i>k$, we can find at least $\rk(k,j)-1$ entries in $H(k,j+1)$ that is larger than $B(k,j+1)$. This implies that $\rk(k,j+1)\geq\rk(k,j)$, a contradiction. Therefore $B(k,j)<B(k,j+1)$ and we are done by induction.
    %If $B(k,j)>B(k,j+1)$, notice that $B(k,j)$ is smaller than $\rk(k,j)-1$ entries in its hook $H(k,j)$. Therefore, we can find the same number of entries in the hook $H(k,j+1)$ that is larger that $B(k,j+1)$, which implies that $\rk(k,j+1)\geq\rk(k,j)$. This contradicts with the fact that $\rk(k,j+1)=\rk(k,j)-1$. Therefore $B(k,j)<B(k,j+1)$ and we are done by induction.
\end{proof}

\begin{cor}\label{cor:same-col-order}
    For any balanced shifted tableaux $B$ of shape $Z(d,r)$ and $1\leq i\leq d$, 
    \[B(i,0)<B(i,1)<B(i,2)<\cdots<B(i,r).\]
\end{cor}
\begin{proof}
    This follows from Lemma~\ref{lm:same-col-order} and the fact that the $0^\text{th}$ column to the $r^\text{th}$ column of $Z(d,r)$ all have the same length $d$.
\end{proof}

\begin{lemma}[Strongly balanced conditions]\label{lm:balance-equiv}
    A shifted tableaux $B$ of shape $Z(d,r)$ is balanced if and only if the following holds at $f^{-1}(\alpha)$:
    %the map $Bf^{-1}$ satisfies the following conditions (called strongly balanced conditions):
    \begin{enumerate}
        \item(Right Staircase) For any $\alpha = e_i-e_j$ where $1\leq j<i\leq d$, %$\alpha=e_i-e_j$,
        %satisfies the following condition
        %\begin{enumerate}
            %\item 
            $Bf^{-1}(\alpha)$ lies between $Bf^{-1}(e_i-e_k)$ and $Bf^{-1}(e_k-e_j)$ for all $j<k<i$;
        %\end{enumerate}
        \item(Rectangle) For any $\alpha=e_i+e_p$ where $1\leq i\leq d<p\leq d+r$, %$\alpha=e_i+e_p$ satisfies the following conditions
        \begin{enumerate}
            \item $Bf^{-1}(\alpha)<Bf^{-1}(e_i+e_q)$ for all $p<q\leq d+r$;
            \item $Bf^{-1}(\alpha)$ lies between $Bf^{-1}(e_i-e_k)$ and $Bf^{-1}(e_k+e_p)$ for all $1\leq k<i$;
        \end{enumerate}
        \item(Column $0$) For any $\alpha=e_i$ where $1\leq i\leq d$, %$\alpha=e_i$ satisfies the following conditions
        \begin{enumerate}
            \item $Bf^{-1}(\alpha)<Bf^{-1}(e_i+e_q)$ for all $d< q\leq d+r$;
            \item $Bf^{-1}(\alpha)$ lies between $Bf^{-1}(e_i-e_k)$ and $Bf^{-1}(e_k)$ for all $1\leq k<i$;
        \end{enumerate}
        \item(Left Staircase) For any $\alpha=e_i+e_j$ where $1\leq j<i\leq d$, % $\alpha=e_i+e_j$ satisfies the following conditions
        \begin{enumerate}
            \item $Bf^{-1}(\alpha)<Bf^{-1}(e_i+e_q)$ for all $d<q\leq d+r$;
            \item $Bf^{-1}(\alpha)<Bf^{-1}(e_j+e_q)$ for all $d<q\leq d+r$;
            \item $Bf^{-1}(\alpha)$ lies between $B\Tilde{f}^{-1}(e_i-e_k)$ and $B\Tilde{f}^{-1}(e_j+e_k)$ for all $-j<k<i$.
        \end{enumerate}
    \end{enumerate}
\end{lemma}
\begin{proof}
    For a root $\alpha\in\Inv(w^{(d,r)})$, denote $\Tilde{H}(\alpha)$ as the set of roots in the extended hook of $\alpha$ in $Z(d,r)$ given by $\Tilde{H}(\alpha)=\Tilde{f}(\Tilde{H}(f^{-1}(\alpha)))$. The set $\Tilde{H}(\alpha)-\{\alpha\}$ can be partitioned into pairs or singletons of roots based on the strongly balanced condition at $f^{-1}(\alpha)$. If the strongly balanced condition at $f^{-1}(\alpha)$ is met, then the hook $\Tilde{H}(\alpha)$ must be balanced. If this is true for all $\alpha$, then $B$ is balanced.
    
    Conversely, assume $B$ is balanced, and we will show that the strongly balanced condition holds at $f^{-1}(\alpha)$ for every $\alpha\in\Inv(w^{(d,r)})$. We will proceed by induction on the size of $\Tilde{H}(\alpha)$.
    %show this in the order of (1), (2), (3), (4), and within each type we will induce on the length of $\Tilde{H}(\alpha)$.
    Although there are eight different statements to prove, it boils down to three cases. We will prove one example for each case as the others hold by similar reasoning. The examples are illustrated in Figure~\ref{fig:proof-lem}.
    
    \emph{Case 1}: statements (1), (2b), (3b), (4c).
    We prove (1) as an example. Consider $\alpha=e_i-e_j$. If $|\tilde{H}(\alpha)| = 1$ or $3$, the statement is clear since $B$ is balanced. Suppose that (1) holds at all $f^{-1}(\beta)$ such that $|\tilde{H}(\beta)|< |\tilde{H}(\alpha)|$ and, for the sake of contradiction, that (1) does not hold at $f^{-1}(\alpha)$. Since $\Tilde{H}(\alpha)$ is balanced, we can find $k,l$ such that $j<k,l<i$ and 
    \[Bf^{-1}(e_i-e_l),Bf^{-1}(e_l-e_j)<x=Bf^{-1}(\alpha)< Bf^{-1}(e_i-e_k), Bf^{-1}(e_k-e_j).\] If $l>k$, we set $y=Bf^{-1}(e_l-e_k)$. By inductive hypothesis, statement (1) holds at $f^{-1}(e_i-e_k)$ and $f^{-1}(e_l-e_j)$. The former implies that $y<x$ and the latter implies that $y>x$, a contradiction. If $k>l$, we set $y = Bf^{-1}(e_k-e_l)$ and a contradiction will be reached by similar reasoning. Therefore we conclude that statement (1) holds. 
    %$Bf^{-1}(e_k-e_j)$ for some $k$. Since $\Tilde{H}(\alpha)$ is balanced, $Bf^{-1}(\alpha)$ must be larger than both $Bf^{-1}(e_i-e_l)$ and $Bf^{-1}(e_l-e_j)$ for some $l$. If $l>k$, we set $y=Bf^{-1}(e_l-e_k)$. 
    
    %WLOG assume $l>k$. Consider $y=Bf^{-1}(e_l-e_k)$. Condition (1) for $\beta=e_i-e_k$ implies that $y<x$; condition (1) for $\beta=e_l-e_j$ implies that $y>x$. This is a contradiction. The proofs for (2b), (3b), (4c) are similar. 
    
    \emph{Case 2}: statements (2a), (3a). They follow directly from Corollary~\ref{cor:same-col-order}.
    
    \emph{Case 3}: statements (4a), (4b).
    We prove (4a) as an example. Here we will induce on $|\tilde{H}(\alpha)|$ for both (4a) and (4b). Set $x=Bf^{-1}(\alpha)$. We start with the base case in the induction and consider $\alpha = e_1+e_2$. Since $B$ is balanced, $x$ is larger than two entries in $\tilde{H}(\alpha)$. By statement (4c), $x$ lies between $B(e_2)$ and $B(e_1)$ as well as between $B(e_2-e_1)$ and $B(2e_1) = B(e_1)$. Therefore $x$ is smaller than $B(e_2+e_q)$ and $B(e_1+e_q)$ for all $d<q\leq d+r$. 
    
    Suppose that (4a) and (4b) holds at all $f^{-1}(\beta)$ such that $|\tilde{H}(\beta)|<|\tilde{H}(\alpha)|$. Suppose further, for the sake of contradiction, that $x$ is larger than $Bf^{-1}(e_i+e_q)$ for some $q$. Since $\Tilde{H}(\alpha)$ is balanced, $x$ must be smaller than both $Bf^{-1}(e_i-e_k)$ and $Bf^{-1}(e_j+e_k)$ for some $k$. If $k\leq 0$, statement (4a) at $f^{-1}(e_i-e_k)$ implies that $Bf^{-1}(e_i-e_k)<Bf^{-1}(e_i+e_q)$, contradicting our inductive hypothesis. If $k>0$, consider $y=Bf^{-1}(e_k+e_q)$. Since Statement (2b) holds at $f^{-1}(e_i+e_q)$,  we have $y<x$. Since statement (4a) and (4b) holds at $f^{-1}(e_j+e_k)$, we get $y>x$, a contradiction. We can then conclude that (4a) holds.
    \begin{figure}[h!]
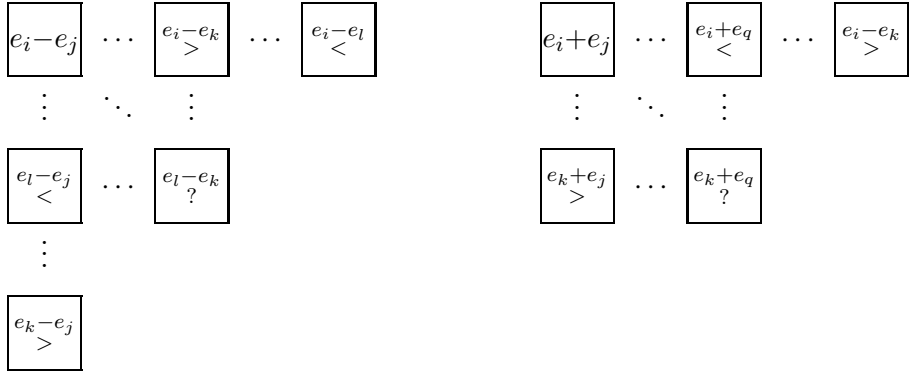

    \centering
    \ytableausetup{boxsize=2.5em}
    \begin{ytableau}
    e_i{-}e_j & \none[\cdots] & \substack{e_i{-}e_k\\>} & \none[\cdots] & \substack{e_i{-}e_l\\<}\\
    \none[\vdots] & \none[\ddots] & \none[\vdots]\\
    \substack{e_l{-}e_j\\<} & \none[\cdots] & \substack{e_l{-}e_k\\?}\\
    \none[\vdots]\\
    \substack{e_k{-}e_j\\>}
    \end{ytableau}
    \hspace{5em}
    \begin{ytableau}
    e_i{+}e_j & \none[\cdots] & \substack{e_i{+}e_q\\<} & \none[\cdots] & \substack{e_i{-}e_k\\>}\\
    \none[\vdots] & \none[\ddots] & \none[\vdots]\\
    \substack{e_k{+}e_j\\>} & \none[\cdots] & \substack{e_k{+}e_q\\?}
    \end{ytableau}
    \caption{Proof ideas for Case 1 and Case 3 of Lemma~\ref{lm:balance-equiv}}
    \label{fig:proof-lem}
    \end{figure}
\end{proof}

\begin{ex}
Let $d=3$ and $r=2$. Figure~\ref{fig:balance-equiv-ex} illustrates an example of the strongly balanced conditions of Lemma~\ref{lm:balance-equiv} at the box $(1,-1)$ (or at root $\alpha=e_3+e_1$), labeled by $\ast$ in the diagram. Its extended hook is colored yellow. The strongly balanced conditions for box $\ast$ are
\begin{itemize}
    \item \underline{Condition (4a) and (4b):} Box $\ast$ is smaller than the boxes labeled by ${+}$;
    \item \underline{Condition (4c):} Box $\ast$ lies between the pair of $a$'s, pair of $b$'s and pair of $c$'s.
\end{itemize}
\begin{figure}[h!]
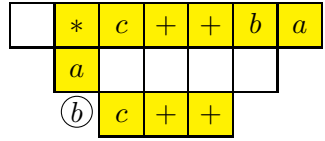

\centering
\ytableausetup{boxsize=1.5em}
\begin{ytableau}
{} & *(yellow)\ast & *(yellow)c & *(yellow){+} & *(yellow){+} & *(yellow)b & *(yellow)a\\
\none & *(yellow)a & {} & {} & {} & {}\\
\none & \none[\circled{b}] & *(yellow)c & *(yellow){+} & *(yellow){+}
\end{ytableau}
\caption{An example of strongly balanced condition for box $\ast$}
\label{fig:balance-equiv-ex}
\end{figure}
\end{ex}
\begin{proof}[Proof of Proposition ~\ref{prop:root-order-balanced}]
    This follows from Lemma~\ref{lm:balance-equiv} since the conditions on $Bf^{-1}$ in the lemma are exactly the conditions for $Bf^{-1}$ being a reflection order in Proposition~\ref{prop:root-ordering}.
\end{proof}

Since there is a natural bijection between $\ro(w)$ and $\Red(w)$,
%the reflection orders in $\ro(w)$ naturally biject with $\Red(w)$,
Proposition ~\ref{prop:root-order-balanced} implies:
\begin{cor}\label{cor:reduced-word-balanced}
The map $B\mapsto \ro^{-1}(Bf^{-1})$ is a bijection between $\BS(Z(d,r))$ and $\Red(w^{(d,r)})$. 
\end{cor}

\begin{ex}
Assume we started with the following balanced tableaux of shape $Z(3,2)$.
\begin{center}
\ytableausetup{boxsize=1.5em}
$B=$
\begin{ytableau}
4 & 8 & 7 & 10 & 13 & 5 & 15\\
\none & 3 & 2 & 6 & 9 & 1\\
\none & \none & 11 & 12 & 14
\end{ytableau}
\end{center}
\end{ex}
The corresponding reflection order $Bf^{-1}$ is given as follows:
\[
\begin{tikzcd}[row sep = small]
12345 \arrow[r, "e_2-e_1"]& 21345 \arrow[r, "e_2"] & \bar21345 \arrow[r, "e_2+e_1"] & 1\bar2345 \arrow[r, "e_3+e_2"] & 13\bar245 \arrow[r, "e_3-e_1"] & 31\bar245\\
\arrow[r, "e_4+e_2"]& 314\bar25 \arrow[r, "e_3"] & \bar314\bar25 \arrow[r, "e_3+e_1"] & 1\bar34\bar25 \arrow[r, "e_5+e_2"] & 1\bar345\bar2 \arrow[r, "e_4+e_3"] & 14\bar35\bar2\\
\arrow[r, "e_1"]& \bar14\bar35\bar2 \arrow[r, "e_4+e_1"] & 4\bar1\bar35\bar2 \arrow[r, "e_5+e_3"] & 4\bar15\bar3\bar2 \arrow[r, "e_5+e_1"] & 45\bar1\bar3\bar2 \arrow[r, "e_3-e_2"] & 45\bar1\bar2\bar3.\\
\end{tikzcd}
\]
Therefore, we can read off a reduced word $\mathbf{a}$ of $w^{(3,2)}=45\bar1\bar2\bar3$ as
\[\mathbf{a}=\ro^{-1}(Bf^{-1})=101213014201324\in\Red(w^{(3,2)}).\]
\section{Bijection between \texorpdfstring{$\SYT(Z(d,r))$}{} and \texorpdfstring{$\Red(w^{(d,r)})$}{} via Kra\'skiewicz's insertion}\label{sec:kraskiewicz}
%\section{Kra\'skiewicz's insertion and the trapezoid case}\label{sec:kraskiewicz}
%The Kra\'skiewicz's insertion algorithm can be viewed as the type $B$ Edelman-Greene insertion.
\subsection{Kra\'skiewicz's insertion algorithm}
We will follow the notations as recorded in Section~1.3 of \cite{Lam1995thesis}. For a shifted tableaux $T$ of shape $\lambda = (\lambda_1,\ldots,\lambda_d)$, define $\pi(T) = T_d T_{d-1},\ldots,T_1$ to be the reading word of $T$ obtained by reading left to right along rows and from bottom to top, where $T_i$ represents the $i$-th row. For a unimodal sequence of integers
\[\mathbf{R} = (r_1>r_2>\ldots>r_k<r_{k+1}<\ldots<r_m),\]
we define the decreasing part of $\mathbf{R}$ to be
\[\mathbf{R}^{\downarrow} = (r_1>r_2>\ldots>r_k),\]
and the increasing part of $\mathbf{R}$ to be
\[\mathbf{R}^{\uparrow} = (r_{k+1}<r_{k+2}<\ldots<r_m).\]
Note that we include the minimal integer of the sequence in $\mathbf{R}^{\downarrow}$.

Let $w\in W(B_n)$ and $\mathbf{a} = a_1a_2\ldots a_{\ell(w)}\in \Red(w)$, we define the \emph{Kra\'skiewicz's insertion algorithm} recursively.
Set $(P^{(0)},Q^{(0)}) := (\emptyset,\emptyset)$, for any $i\in [\ell(w)]$, define the insertion
\[(P^{(i-1)},Q^{(i-1)}) \leftarrow a_i =: (P^{(i)},Q^{(i)})\]
%of $a_i$ into $(P^{(i-1)},Q^{(i-1)})$
as follows:

Step 1: Set $\mathbf{R}$ to be the first row of $P^{(i-1)}$ and $a = a_i$.

Step 2: Insert $a$ into $\mathbf{R}$ as follows:
\begin{itemize}
    \item Case 0 ($\mathbf{R} = \emptyset$): Insert $a$ into the left-most box of the row to obtain $P^{(i)}$. Then define $Q^{(i)}$ from $Q^{(i-1)}$ by adding $i$ to the unique box in $P^{(i)}/P^{(i-1)}$. Stop.
    \item Case 1 ($\mathbf{R}a$ is unimodal): Append $b_i$ to the right of $\mathbf{R}$ and to obtain $P^{(i)}$. Then define $Q^{(i)}$ from $Q^{(i-1)}$ by adding $i$ to the unique box in $P^{(i)}/P^{(i-1)}$. Stop.
    \item Case 2 ($\mathbf{R}a$ is not unimodal): Let $b$ be the smallest number in $\bf{R}^{\uparrow}$ such that $b\geq a$.
        \begin{itemize}
            \item Case 2.0 ($a=0$ and $\mathbf{R}$ contains $101$ as a subsequence): We leave $\mathbf{R}$ unchanged and return to start of Step 2 with $a= 0$ and $\mathbf{R}$ equals the next row.
            \item Case 2.1.1 ($b\neq a$): Replace $b$ with $a$ and set $c=b$. 
            \item Case 2.1.2 ($b=a$): Keep $\mathbf{R}^{\uparrow}$ unchanged and set $c = a+1$.
        \end{itemize}
    We now insert $c$ into $\mathbf{R}^{\downarrow}$. Let $d$ be the largest integer such that $d\leq c$. This number always exists since $\mathbf{R}^{\downarrow}$ contains the smallest number in the row. 
        \begin{itemize}
            \item Case 2.1.3 ($d\neq c$): Replace $d$ with $c$ and set $a' = d$.
            \item Case 2.1.4 ($d=c$): Keep $\bf{R}^{\downarrow}$ unchanged and set $a' = c-1$.
        \end{itemize}
\end{itemize}

Step 3: Repeat Step 2 with $a = a'$ and $\bf{R}$ the next row.

Define $P(\mathbf{a}) = P^{(\ell(w))}$ to be the \emph{insertion tableau} and $Q(\mathbf{a}) = Q^{(\ell(w))}$ to be the \emph{recording tableau}. 

\begin{ex}
Let $w = w^{(3,2)} = 45\bar{1}\bar{2}\bar{3}$ as in Example~\ref{ex:45123}. Consider the reduced word $\mathbf{a} = 010121012342312 \in \Red(w)$. Following the above insertion algorithm, we obtain
\begin{center}
\ytableausetup{boxsize=1em}
$P^{(0)} = \emptyset \xrightarrow{0}$
\begin{ytableau}
0
\end{ytableau}
$\xrightarrow{1}$ 
\begin{ytableau}
0 & 1
\end{ytableau}
$\xrightarrow{0}$
\begin{ytableau}
1 & 0\\
\none & 0
\end{ytableau}
$\xrightarrow{1}$
\begin{ytableau}
1 & 0 & 1\\
\none & 0
\end{ytableau}
$\xrightarrow{2}$
\begin{ytableau}
1 & 0 & 1 & 2\\
\none & 0
\end{ytableau}
$\xrightarrow{1}$
\begin{ytableau}
2 & 0 & 1 & 2\\
\none & 0 & 1
\end{ytableau}
$\xrightarrow{0}$
\begin{ytableau}
2 & 1 & 0 & 2\\
\none & 1 & 0\\
\none & \none &0
\end{ytableau}
\vspace{0.2cm}
$\xrightarrow{1}$
\begin{ytableau}
2 & 1 & 0 & 1\\
\none & 1 & 0 & 1\\
\none & \none &0
\end{ytableau}
$\xrightarrow{2}$
\begin{ytableau}
2 & 1 & 0 & 1 & 2\\
\none & 1 & 0 & 1\\
\none & \none &0
\end{ytableau}
\vspace{0.2cm}
$\xrightarrow{3}$
\begin{ytableau}
2 & 1 & 0 & 1 & 2 & 3\\
\none & 1 & 0 & 1\\
\none & \none &0
\end{ytableau}
$\xrightarrow{4}$
\begin{ytableau}
2 & 1 & 0 & 1 & 2 & 3 & 4\\
\none & 1 & 0 & 1\\
\none & \none &0
\end{ytableau}
$\xrightarrow{2}$
\begin{ytableau}
3 & 1 & 0 & 1 & 2 & 3 & 4\\
\none & 1 & 0 & 1 & 2\\
\none & \none &0
\end{ytableau}
$\xrightarrow{3}$
\begin{ytableau}
4 & 1 & 0 & 1 & 2 & 3 & 4\\
\none & 1 & 0 & 1 & 2 & 3\\
\none & \none &0
\end{ytableau}
$\xrightarrow{1}$
\begin{ytableau}
4 & 2 & 0 & 1 & 2 & 3 & 4\\
\none & 2 & 0 & 1 & 2 & 3\\
\none & \none &0 & 1
\end{ytableau}
$\xrightarrow{2}$
\begin{ytableau}
4 & 3 & 0 & 1 & 2 & 3 & 4\\
\none & 3 & 0 & 1 & 2 & 3\\
\none & \none &0 & 1 & 2
\end{ytableau}.
\end{center}
We then have
\[\ytableausetup{boxsize=1.1em}
P(\mathbf{a}) = \begin{ytableau}
4 & 3 & 0 & 1 & 2 & 3 & 4\\
\none & 3 & 0 & 1 & 2 & 3\\
\none & \none &0 & 1 & 2
\end{ytableau},\ 
Q(\mathbf{a}) = \begin{ytableau}
1 & 2 & 4 & 5 & 9 & 10 & 11\\
\none & 3 & 6 & 8 & 12 & 13\\
\none & \none &7 & 14 & 15
\end{ytableau},
\]
and the reading word $\pi(P(\mathbf{a})) = 012301234301234$.
\end{ex}

\begin{defin}\label{def:SDT}
    For a shifted tableaux $T$ with $m$ rows, we say $T$ is a \emph{standard decomposition tableaux} of $w\in W(B_n)$ if
    \begin{enumerate}
        \item $\pi(T) = T_m T_{m-1},\ldots,T_1$ is a reduced word of $w$,
        \item $T_{i}$ is a unimodal subsequence of maximal length in $T_{m}T_{m-1}\ldots T_{i}$.
    \end{enumerate}
Define the set of all such tableaux to be $\SDT(w)$.
\end{defin}

\begin{theorem}[Theorem~5.2, \cite{K89}]\label{thm:insertion}
    The Kra\'skiewicz's insertion gives a bijection between $\{\mathbf{a}\in \Red(w)\}$ and the pairs of tableaux $(P(\mathbf{a}),Q(\mathbf{a}))$ where $P(\mathbf{a})\in \SDT(w)$ and $Q(\mathbf{a})$ is a standard tableaux of the same shape.
\end{theorem}
\begin{cor}[Section~6, \cite{K89}]
For any $w\in W(B_n)$,
\[|\Red(w)| = \sum_{P\in \SDT(w)} f^{\sh(P)}.\]
\end{cor}

\subsection{Type C Stanley symmetric functions and vexillary elements}
For $w\in W(B_n)$ and any strict partition $\lambda$, let $F^{C}(w)$ and $Q_{\lambda}$ be the corresponding \emph{type C Stanley symmetric function} and
%define the \emph{type C Stanley symmetric function} by
%\[F^C(w) = \sum_{\mathbf{a}\in \Red(w)}\sum_{\substack{1\leq i_1\leq i_2\leq \cdots\leq i_\ell,\\ a_j<a_{j+1}>a_{j+2}\implies i_{j}<i_{j+2}}}2^{|\mathbf{i}|}x_{i_1}x_{i_2}\cdots x_{i_\ell},\]
%where $|\mathbf{i}|$ is the number of \emph{different} entries in $i_1,\cdots,i_\ell$.
\emph{Schur-Q function} respectively. See \cite{BLvexillary} for the exact definitions. 
\begin{theorem}[Theorem~3.12, \cite{Lam96}]
    $F^{C}(w) = \sum_{T\in \SDT(w)}Q_{\sh(T)}$.
\end{theorem}

\begin{defin}
    A permutation $w\in W(B_n)$ is said to be \emph{vexillary} if $\SDT(w)$ consists of exactly one shifted tableau. We denote this tableau as $P(w)$.
\end{defin}

\begin{defin}
    For any $w\in W(B_n)$ and any $v\in W(B_m)$ such that $m\leq n$, we say $w$ \emph{pattern embeds} $v$ if the following is true for some $1\leq i_1<i_2<\ldots<i_m\leq n$:
    \begin{enumerate}
        \item $w(i_j)$ has the same sign as $v(j)$,
        \item For all $j, k$, $|w(i_j)|<|w(i_k)|$ if and only if $|v(j)|<|v(k)|$.
    \end{enumerate}
We say $w$ \emph{pattern avoids} $v$ if $w$ does not pattern embed $v$.
\end{defin}

\begin{theorem}[Theorem~7, \cite{BLvexillary}]
    A permutation $w\in W(B_n)$ is vexillary if and only if $w$ pattern avoids the following permutations:
    \[\begin{matrix}
        \bar{3}2\bar{1} & \bar{3}21 & 32\bar{1} & 321 & 3\bar{1}2 &
        \bar{2}31 & \bar{1}32 & \bar{4}\bar{1}\bar{2}3 & \bar{4}1\bar{2}3 \\ \bar{3}\bar{4}\bar{1}\bar{2} &
        \bar{3}\bar{3}1\bar{2} & 3\bar{4}\bar{1}\bar{2} & 3\bar{4}1\bar{2} & 3142 & \bar{2}\bar{3}4\bar{1} &
        2413 & 2\bar{3}4\bar{1} & 2143 
    \end{matrix}\,.\]
\end{theorem}
By comparing \eqref{eqn:defw} with the above list of patterns, we directly see that $w^{(d,r)}$ is vexillary. In particular, this means that any $\mathbf{a}\in \Red(w^{(d,r)})$ has the same $P$ tableau.

\subsection{Proof of Theorem~\ref{thm:main} in the case \texorpdfstring{$\lambda = Z(d,r)$}{}}

\begin{prop}\label{prop:shape-of-P(w)}
For $d>0$ and $r\geq0$, $\sh(P(w^{(d,r)})) = Z(d,r)$. Moreover, \begin{equation*}
        \begin{split}
            P(w^{(d,r)})(i,j) = 
            \begin{cases}
            r-j &\text{ if }j<0,\\
            j &\text{ if }j\geq 0.
            \end{cases}
        \end{split}
    \end{equation*}
\end{prop}
See Figure~\ref{fig:trapezoidT} for this tableaux. A concrete example is also shown in Example~\ref{ex:w-lam}.
\begin{proof}
Consider the shifted tableau $T$ defined as above in the proposition statement. We claim that $T = P(w^{(d,r)})$. Since $w^{(d,r)}$ is vexillary, it is enough to show that $T\in \SDT(w)$.

For every $i\in [d]$, it is clear from definition that
\[T_i = r+d-i,\ldots,r+1,0,1,\ldots,r+d-i\]
is a unimodal subsequence of maximal length in $T_dT_{d-1}\ldots T_1$. For $i\in [d]$, define $u^{(i)}\in W(B_n)$ by the one-line notation:
\[u^{(i)} = (i+1,i+2,\ldots,r+i,\bar{1},\ldots,\bar{i},r+i+1,\ldots,r+d).\]
It is straightforward to check that, for all $i\in[d]$,
\[T_dT_{d-1}\ldots T_{d-i+1}\in \Red(u^{(i)}).\]
Notice that $w^{(d,r)} = u^{(d)}$. We conclude that $\pi(T)\in \Red(w^{(d,r)})$ and thus $T\in \SDT(w)$.
\end{proof}
\begin{figure}[h!]
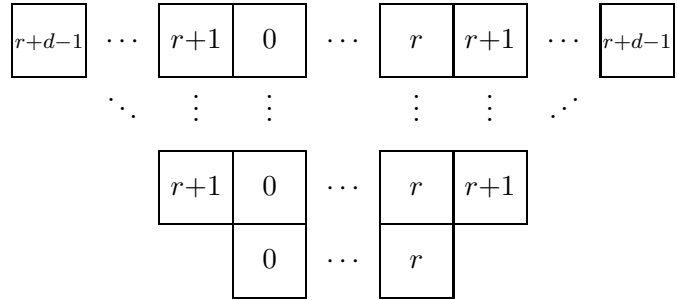

\centering
\ytableausetup{boxsize=2.5em}
\begin{ytableau}
\scriptstyle{r{+}d{-}1} & \none[\cdots] & r{+}1 & 0 & \none[\cdots] & r & r{+}1& \none[\cdots] & \scriptstyle{r{+}d{-}1}\\
\none & \none[\ddots] & \none[\vdots] & \none[\vdots] & \none & \none[\vdots] &\none[\vdots] & \none[\iddots]\\
\none & \none & r{+}1 & 0 & \none[\cdots] & r & r{+}1\\
\none & \none & \none & 0 & \none[\cdots] & r
\end{ytableau}
\caption{The insertion tableaux $P(w^{(d,r)})$ of shape $Z(d,r)$}
\label{fig:trapezoidT}
\end{figure}

By Theorem~\ref{thm:insertion} and Proposition~\ref{prop:shape-of-P(w)}, we obtain:
\begin{cor}\label{cor:uniquePtableau}
By restricting to the recording tableaux, Kra\'skiewicz's insertion gives a bijection between $\Red(w^{(d,r)})$ and $\SYT(Z(d,r))$.
\end{cor}

Combining Corollary~\ref{cor:reduced-word-balanced} and Corollary~\ref{cor:uniquePtableau}, we derive Theorem~\ref{thm:main} for the trapezoid shape $Z(d,r)$.
\section{The general case}\label{sec:general}
In this Section, we prove our main result, Theorem~\ref{thm:main}, in full generality. Fix a strict partition $\lambda\subset Z(d,r)$ such that $\lambda_d>0$ and set $N = |\lambda|$. Let $\ell = |Z(d,r)| = \ell(w^{(d,r)})$ and set $\mu_0 = \sigma_0 = 0$, $\mu_i = Z(d,r)_i-\lambda_i$ and $\sigma_i = \sum_{k=1}^{i}\mu_k$ for all $i\in [d]$. Recall our main framework \[\SYT(\lambda)\longleftrightarrow\SYT(Z(d,r))|_{\lambda}\longleftrightarrow\Red(w^{\lambda})\longleftrightarrow\BS(Z(d,r))|_{\lambda}\longleftrightarrow\BS(\lambda).\] The first arrow of bijection is immediate (Definition~\ref{def:SYT|lambda}). In each of the subsequent subsections, we prove one remaining bijection respectively. We will make heavy use of Section~\ref{sec:kraskiewicz}, and consider $\Red(w^{\lambda})$ as the subset of $\Red(w^{(d,r)})$ that ends with a fixed sequence.

A complete example is given at the end of this section in Example~\ref{ex:main-ex}. Readers are encouraged to refer to this example for intuition. 
\subsection{Bijection between \texorpdfstring{$\SYT(Z(d,r))|_{\lambda}$}{} and \texorpdfstring{$\Red(w^{\lambda})$}{}}
\begin{defin}\label{def:SYT|lambda}
    For any tableau $T\in \SYT(\lambda)$, define $T^{+}\in \SYT(Z(d,r))$ to be the tableau obtained from $T$ by assigning $N+1,\ldots,\ell$ to the cells in $Z(d,r)\setminus \lambda$ from left to right along rows and from top to bottom. Define $\SYT(Z(d,r))|_{\lambda}$ to be the set of all such $T^+$ obtained from some $T\in \SYT(\lambda)$.
\end{defin}
\begin{ex}
Let $\lambda = (6,2,1)\subset Z(3,2)$ and 
    \[\ytableausetup{boxsize=1.5em}
T = \begin{ytableau}
1 & 2 & 3 & 5 & 6 & 9 \\
\none & 4 & 7 & \none & \none \\
\none & \none & 8 & \none 
\end{ytableau}, \text{ then }
T^{+} = \begin{ytableau}
1 & 2 & 3 & 5 & 6 & 9 & \circled{10}\\
\none & 4 & 7 & \circled{11} & \circled{12} & \circled{13}\\
\none & \none & 8 & \circled{14} & \circled{15}
\end{ytableau}\, .
\]
\end{ex}
%%Define w^{\lambda}
\begin{lemma}[Lemma~1.25, \cite{Lam1995thesis}]\label{lemma:reverse-insertion}
Given $(P(\mathbf{a}),Q(\mathbf{a}))$ for some $\mathbf{a} = a_1a_2\cdots a_{\ell(w)}\in \Red(w)$ and $w\in W(B_n)$, let $Q'$ be obtained by removing the largest entry in $Q(\mathbf{a})$. Then there is a unique $a\in [0,n-1]$ and a unique $P'\in \SDT(ws_{a})$ such that $P' \leftarrow a = P,$ and $\sh(P') = \sh(Q')$. In fact, we have $a = a_{\ell(w)}$.
\end{lemma}

\begin{prop}\label{prop:unique-inverse}
There exists a unique $w^{\lambda}\in W(B_n)$ and a unique word 
\[\mathbf{a^\lambda} = a_{N+1}\cdots a_{\ell}\in \Red((w^{\lambda})^{-1}w^{(d,r)})\]
such that for every $T^+\in \SYT(Z(d,r))|_{\lambda}$, there is a unique $\mathbf{a'}\in \Red(w^{\lambda})$ so that $T^+ = Q(\mathbf{a'}\mathbf{a^\lambda})$.
\end{prop}
\begin{proof}
We show that under the bijection of Corollary~\ref{cor:uniquePtableau}, every $T^+\in \SYT(Z(d,r))|_{\lambda}$ corresponds to a reduced word with a fixed ending sequence of length $\ell-N$, establishing the uniqueness of $\mathbf{a}^{\lambda}$, and consequently of $w^{\lambda}$. 
Let $T^+,S^+\in \SYT(Z(d,r))|_{\lambda}$. By Corollary~\ref{cor:uniquePtableau}, there is a unique $\mathbf{a} = a_1 a_2 \cdots a_{\ell}\in \Red(w^{(d,r)})$ and a unique $\mathbf{b} = b_1b_2\cdots b_{\ell}\in \Red(w^{(d,r)})$ such that $Q(\mathbf{a}) = T^+$ and $Q(\mathbf{b}) = S^+$.
Let $Q'(\mathbf{a})$ and $Q'(\mathbf{b})$ be obtained from $Q(\mathbf{a})$ and $Q(\mathbf{b})$ by removing the largest entry respectively.
Since $i$ appears in the same box in $T^{+}$ and $S^+$ for all $i>N$, $\sh(Q'(\mathbf{a})) = \sh(Q'(\mathbf{b}))$.
Since $P(\mathbf{a}) = P(\mathbf{b})$, by the uniqueness of $a$ and $P'$ in Lemma~\ref{lemma:reverse-insertion}, we get $a_{\ell}  = b_{\ell}$ and $P(a_1\cdots a_{\ell-1}) = P(b_1\cdots b_{\ell-1}).$

Note that the above argument holds as long as the two insertion tableaux are the same and the largest entry in the two recording tableaux appear in the same box. In particular, we can apply the argument $\ell-N$ times and get
\begin{equation}\label{eqn:Mar19aaa}
    b_i = a_i\ \text{for all }i>N \text{ and }P(a_1\cdots a_N) = P(b_1\cdots b_N).
\end{equation}
Set
\begin{equation}\label{eqn:w^lambda}
    \mathbf{a^\lambda} = a_{N+1}\cdots a_{\ell}\ \text{and }w^{\lambda} = w^{(d,r)}s_{a_{\ell}}\cdots s_{a_{N+1}}.
\end{equation}
Then for every $\mathbf{a}\in \Red(w^{(d,r)})$ such that $Q(\mathbf{a})\in \SYT(Z(d,r))|_{\lambda}$, 
$\mathbf{a} = \mathbf{a'}\mathbf{a^\lambda}$
for some $\mathbf{a'}\in \Red(w^{\lambda})$. The uniqueness of such $\mathbf{a'}$ follows from Corollary~\ref{cor:uniquePtableau}.
\end{proof}
\begin{cor}\label{cor:RedtoSYTlambda}
    Let $\mathbf{a^\lambda}$ and $w^{\lambda}$ be as defined in \eqref{eqn:w^lambda}. %Proposition~\ref{prop:alambda-and-Plambda}
     Then there is a bijection between $\Red(w^{\lambda})$ and $\SYT(Z(d,r))|_{\lambda}$ given by $\mathbf{a'} \mapsto Q(\mathbf{a'}\mathbf{a^\lambda})$.
\begin{comment}
    \begin{align}\label{eqn:Mar19abc}
    \begin{split}
        \Red(w^{\lambda})\to &\; \SYT(Z(d,r))|_{\lambda}\\
        \mathbf{a'} \mapsto &\; Q(\mathbf{a'}\mathbf{a^\lambda}).
    \end{split}
    \end{align}
\end{comment}
\end{cor}
\begin{proof}
    By \eqref{eqn:Mar19aaa}, the $P$ tableaux are the same for every $\mathbf{a'}\in\Red(w^{\lambda})$. Since  $\sh(P(\mathbf{a'})) = \lambda$, $Q(\mathbf{a'}\mathbf{a^{\lambda}})\in \SYT(Z(d,r))|_{\lambda}$ and the map is well-defined. Injectivity follows from the injectivity of Kra\'skiewicz's insertion, and surjectivity follows from Proposition~\ref{prop:unique-inverse}.\qedhere
\end{proof}
Denote the unique $P$ tableau of $\Red(w^{\lambda})$ as $P(w^{\lambda})$. The next Proposition provides explicit description of $\mathbf{a}^{\lambda}$ and $P(w^{\lambda})$. 
\begin{prop}\label{prop:alambda-and-Plambda}
Let $\mathbf{a^\lambda}$ and $w^{\lambda}$ be as defined in \eqref{eqn:w^lambda}. Then $\mathbf{a^\lambda} = \mathbf{a_1^\lambda}\cdots \mathbf{a_d^\lambda}$ where
\begin{equation}\label{eqn:Mar21aaa}
    \mathbf{a_i^\lambda} = d+r-i-\mu_i+1,\ldots, d+r-i.
\end{equation}
This is simply reading off the entries in $Z(d,r)\setminus \lambda$ in Figure~\ref{fig:trapezoidT} from left to right along rows and from top to bottom. In terms of $P = P(w^\lambda)$,
\begin{equation}\label{eqn:Mar19bbb}
    P(i,j) = \begin{cases}
    r-j-\nu_{d+1+j} & \text{ if }j<0\\
    j & \text{ if }j\geq 0,
\end{cases}
\end{equation}
where $\nu_i=\min\{\mu_i,\mu_{i+1},\dots,\mu_d\}$ for all $i\in [d]$.
\end{prop}

%% We can add the description of \mathbf{a'} here.
% TODO: Characterize $P(w^{\lambda})$ and $\mathbf{a'}$. I believe (check me) $\mathbf{a'} = \mathbf{a'_1}\cdots \mathbf{a'_d}$ where
% \[\mathbf{a'_i} = d-i+1+r-a_i,\ldots, d-i+r.\]
% This is simply reading off the entries in $Z(d,r)\setminus \lambda$ in Figure~\ref{fig:trapezoidT} from left to right along rows and from top to bottom. 

% In terms of $P = P(w^{\lambda})$, set $c_d = a_d$ and $c_{i} = \min(a_i,c_{i+1})$. I claim that (check me)
% \[P(i,j) = \begin{cases}
%     r-j-c_{d+1+j} & \text{ if }j<0\\
%     j & \text{ if }j\geq 0.
% \end{cases}
% \]
\begin{proof}
Recall from the proof of Proposition~\ref{prop:unique-inverse} that $\mathbf{a}^{\lambda}$ is obtained in the following way: let $P^+$ be the unique $P$ tableau corresponding to $\Red(w^{(d,r)})$ and $Q^+\in\SYT(Z(d,r))|_{\lambda}$, then $\ell-N$ steps of reverse Kra\'skiewicz's insertion on $(P^+,Q^+)$ pop out $a_{\ell},\ldots,a_{N+1}$. Readers are referred to Example~\ref{ex:w-lam} for visualization.

We use induction on $\ell-N$. The base case $\lambda=Z(d,r)$ is done in Proposition~\ref{prop:shape-of-P(w)}. For a general $\lambda$ with $N<\ell$, let $x$ be the smallest row index such that $\lambda_x\neq Z(d,r)_x$. Let $\lambda'=(\lambda_1,\ldots,\lambda_x+1,\ldots,\lambda_d)\subset Z(d,r)$ be obtained from $\lambda$ by adding one to row $x$. Let $(x,y)$ be the coordinate of this added box, where $y=\lambda_x+x-d=r+d+1-x-\mu_x>0$. Let $P$ be as described in \eqref{eqn:Mar19bbb}, and $P'=P(w^{\lambda'})$ as described analogously by inductive hypothesis. Notice that for $Q^+\in\SYT(Z(d,r))|_{\lambda}$, the largest entry of $Q':=Q^+|_{\lambda'}$ must be at coordinate $(x,y)$. In order to show that one step of reverse Kra\'skiewicz's insertion from $(P',Q')$ pops out $y$ and results in $(P,Q)$, by Lemma~\ref{lemma:reverse-insertion}, it suffices to show that 
\[P\leftarrow y=P',\]
and that $P$ is a standard decomposition tableaux (Definition~\ref{def:SDT}). Analogously define $\mu_i'$ and $\nu_i'$ for $\lambda'$. Note that $\mu_i = \mu_i' = 0$ for all $i\in [x-1]$
%$\mu_1=\cdots=\mu_{x-1}=0$, $\mu_1'=\cdots=\mu_{x-1}'=0$ 
and $\mu_x=\mu_x'+1$. Also recall the notations in Section~\ref{sec:kraskiewicz} that the rows of $P$ are labeled as $P_1,\ldots,P_d$ and the reading word for $P$ is $\pi(P)=P_d\cdots P_1$. Let $w(P)$ be the signed permutation corresponding to this word.

We divide into two cases and show that $P\leftarrow y=P'$ and $w(P)s_y=w(P')$.

\emph{Case 1}: $\nu_x=\mu_x$. Let $j=x-d-1$. Then column $j$ of $P$ has value $r-j-\mu_x=y$ and column $j-1$ of $P$ has value $r-(j-1)-\mu_{x-1}=r-(j-1)\geq y+2$. In this case, $\nu$ and $\nu'$ differ only at $\nu_x=\mu_x$ and $\nu_x'=\mu_x'=\mu_x-1$. So the columns of $P'$ with negative indices can be obtained from that of $P$ by increasing column $j$ by $1$, that is, changing all the values $y$ to $y+1$. We then analyze the situations where $y$ is inserted into $P$, going through the algorithm in Section~\ref{sec:kraskiewicz}. In row $i$ where $i\leq x-1$, as $y$ cannot be the largest entry in $P^{\uparrow}$, we go to case 2.1.2 and insert $y+1$ to $P^{\downarrow}$; as $y$ is the largest integer weakly smaller than $y+1$, we go to case 2.1.3 by replacing $y$ with $y+1$ and continue to insert $y$ to the next row. During this process for $i\leq x-1$, we see that 
\allowdisplaybreaks
\begin{align*}
w(P_i)s_y=&w(P_i\downarrow)s_1s_2\cdots s_{\lambda_i-d+i-1}s_y\\
=&w(P_i\downarrow)s_{y+1}s_1s_2\cdots s_{\lambda_i-d+i-1}=w(P_i\downarrow)s_{y+1}w(P_i\uparrow)\\
=&\cdots s_{r-j+1}s_{y}\cdots s_{y+1}w(P_i\uparrow)\\
=&s_y\cdots s_{r-j+1}s_{y+1}\cdots w(P_i\uparrow)\\
=&s_yw(P'_i\downarrow)w(P_i\uparrow)=s_yw(P_i')
\end{align*}
where we used Coxeter relations $s_ys_{y+1}s_y=s_{y+1}s_{y}s_{y+1}$ and the commutation relations as $r-j+1\geq y+2$.
Finally, in row $x$, we insert $y$ to the rightmost position at coordinate $(x,y)$, resulting in $P'$ as desired. This also says that $w(P_x)s_y=w(P_x')$. Together with the above equalities, we obtain $w(P)s_y=w(P')$.

\emph{Case 2}: $\nu_x<\mu_x$. Let $\nu_x=\mu_z<\mu_x$ with $z>x$. In this case, $\nu_x'=\min\{\mu_x-1,\mu_z\}=\mu_z=\nu_x$ so $\nu$ and $\nu'$ are identical, and this means that the tableaux $P$ and $P'$ are the same in the negative columns. Let $j=r-\mu_z-y-1$. Then $d+1+j=d+r-\mu_z-y=d+r-\mu_z-(d+r+1-x-\mu_x)=x+\mu_x-\mu_z-1$. As $\mu_z<\mu_x$, this value $d+1+j\geq x$ and as $\lambda$ is a strict partition, $\mu_x+x\leq\mu_z+z$ so $d+1+j\leq z-1$. This calculation shows that column $j<0$ exists in row $1$ through row $x$ of $P$. Moreover, as $x\leq d+1+j\leq z$, $\nu_{d+1+j}=\nu_z$ so column $j$ of $P$ takes on value $r-j-\nu_z=y+1$. Since $d+2+j\leq z$, column $j+1$ of $P$ takes on value $r-(j+1)-\nu_{d+2+j}=y$. Now, for $i=1,\ldots,x-1$, when we insert $y$ to row $i$ of $P$, we go to case 2.1.2 as before since $y$ cannot be the largest entry of $P^{\uparrow}$ to insert $y+1$ to $P^{\downarrow}$; since $y+1$ exists in $P^{\downarrow}$ of row $i$, we go to case 2.1.4 by keeping $P^{\downarrow}$ unchanged and continue to insert $y$ to the next row. During this process, for $i\leq x-1$,
\allowdisplaybreaks
\begin{align*}
w(P_i)s_y=&w(P_i\downarrow)s_1s_2\cdots s_{\lambda_i-d+i-1}s_y\\
=&w(P_i\downarrow)s_{y+1}s_1s_2\cdots s_{\lambda_i-d+i-1}=w(P_i\downarrow)s_{y+1}w(P_i\uparrow)\\
=&\cdots s_{y+1}s_{y}\cdots s_{y+1}w(P_i\uparrow)=\cdots s_{y+1}s_{y}s_{y+1}\cdots w(P_i\uparrow)\\
=&\cdots s_ys_{y+1}s_y\cdots w(P_i\uparrow)=s_y\cdots s_{y+1}s_y\cdots w(P_i\uparrow)\\
=&s_yw(P_i\downarrow)w(P_i\uparrow)=s_yw(P_i)=s_yw(P'_{i}).
\end{align*}
Nothing changes until in the end, we insert $y$ in row $x$ at coordinate $(x,y)$, to obtain $P'$ as desired. This gives $w(P_x)s_y=w(P'_x)$ so $w(P)s_y=w(P')$ as desired.

We then show that $P$ is a standard decomposition tableaux by checking the two conditions in Definition~\ref{def:SDT}, where condition (2) is clear from construction. For (1), we use induction hypothesis that $\pi(P')$ is reduced, i.e. $\ell(\pi(P'))=|\lambda'|=|\lambda|+1$. As $w(P)s_y=w(P')$, $\ell(w(P))\geq\ell(w(P'))-1=|\lambda|$ so $\pi(P)$ is reduced as well.
\end{proof}

\begin{ex}\label{ex:w-lam}
Consider $\lambda=(6,2,1)\subset Z(3,2)$. Here $w^{(3,2)}=45\bar1\bar2\bar3$. We compute $\mathbf{a^\lambda}$ and $w^{\lambda}$ defined in \eqref{eqn:w^lambda} step by step via (the reverse of) Kra\'skiewicz's insertion. We start with the unique $P^+$ (see Proposition~\ref{prop:shape-of-P(w)}), and a standard shifted tableau $Q^+$ of shape $Z(3,2)$, padded from any standard shifted tableau $Q$ of shape $\lambda$, shown in Table~\ref{tab:ex-w-lam}. 

\begin{table}[h!]
\centering
\setlength{\extrarowheight}{10pt}
\ytableausetup{smalltableaux}
\begin{tabular}{c|c|c}
\text{insertion tableau} $P$ & \text{recording tableau} $Q$ & letter $a_i$'s \\
\begin{ytableau}
4 & 3 & 0 & 1 & 2 & 3 & 4 \\
\none & 3 & 0 & 1 & 2 & 3\\
\none & \none & 0 & 1 & 2
\end{ytableau} & 
\begin{ytableau}
*(yellow) & *(yellow) & *(yellow) & *(yellow) & *(yellow) & *(yellow) & 10\\
\none & *(yellow) & *(yellow) & 11 & 12 & 13\\
\none & \none & *(yellow) & 14 & 15
\end{ytableau} &  \\
\begin{ytableau}
4 & 2 & 0 & 1 & 2 & 3 & 4 \\
\none & 2 & 0 & 1 & 2 & 3\\
\none & \none & 0 & 1
\end{ytableau} & 
\begin{ytableau}
*(yellow) & *(yellow) & *(yellow) & *(yellow) & *(yellow) & *(yellow) & 10\\
\none & *(yellow) & *(yellow) & 11 & 12 & 13\\
\none & \none & *(yellow) & 14
\end{ytableau} & $a_{15}=2$ \\
\begin{ytableau}
4 & 1 & 0 & 1 & 2 & 3 & 4 \\
\none & 1 & 0 & 1 & 2 & 3\\
\none & \none & 0
\end{ytableau} & 
\begin{ytableau}
*(yellow) & *(yellow) & *(yellow) & *(yellow) & *(yellow) & *(yellow) & 10\\
\none & *(yellow) & *(yellow) & 11 & 12 & 13\\
\none & \none & *(yellow)
\end{ytableau} & $a_{14}=1$ \\
\begin{ytableau}
3 & 1 & 0 & 1 & 2 & 3 & 4 \\
\none & 1 & 0 & 1 & 2 \\
\none & \none & 0
\end{ytableau} & 
\begin{ytableau}
*(yellow) & *(yellow) & *(yellow) & *(yellow) & *(yellow) & *(yellow) & 10\\
\none & *(yellow) & *(yellow) & 11 & 12\\
\none & \none & *(yellow)
\end{ytableau} & $a_{13}=3$ \\
\begin{ytableau}
2 & 1 & 0 & 1 & 2 & 3 & 4 \\
\none & 1 & 0 & 1  \\
\none & \none & 0
\end{ytableau} & 
\begin{ytableau}
*(yellow) & *(yellow) & *(yellow) & *(yellow) & *(yellow) & *(yellow) & 10\\
\none & *(yellow) & *(yellow) & 11 \\
\none & \none & *(yellow)
\end{ytableau} & $a_{12}=2$ \\
\begin{ytableau}
2 & 1 & 0 & 1 & 2 & 3 & 4 \\
\none & 1 & 0   \\
\none & \none & 0
\end{ytableau} & 
\begin{ytableau}
*(yellow) & *(yellow) & *(yellow) & *(yellow) & *(yellow) & *(yellow) & 10\\
\none & *(yellow) & *(yellow) \\
\none & \none & *(yellow)
\end{ytableau} & $a_{11}=1$ \\
\begin{ytableau}
2 & 1 & 0 & 1 & 2 & 3 \\
\none & 1 & 0   \\
\none & \none & 0
\end{ytableau} & 
\begin{ytableau}
*(yellow) & *(yellow) & *(yellow) & *(yellow) & *(yellow) & *(yellow) \\
\none & *(yellow) & *(yellow) \\
\none & \none & *(yellow)
\end{ytableau} & $a_{10}=4$ \\
\end{tabular}
\caption{An example of $\mathbf{a}^{\lambda}$ and $w^{\lambda}$}
\label{tab:ex-w-lam}
\end{table}
We read off $\mathbf{a}^{\lambda}=a_{10}\cdots a_{15}=412312$. As $w^{\lambda}=w^{(d,r)}s_{a_{15}}\cdots s_{a_{10}}$, we obtain $w^{\lambda}=\bar2\bar14\bar35$.
\end{ex}

\begin{remark}
Here is an explicit description for $w^\lambda$. Draw a $(d+r)\times(d+r)$ triangle over $\lambda$ with box $(1,1)$ as the top left corner. Rotate by $45^\circ$ and consider the Dyck path of semilength $d+r$ that is the border of $\lambda$. Label the upsteps with $d+r,d+r-1,\dots,d+1,\bar1,\bar2,\dots,\Bar{d}$ in order, and pass these labels to the corresponding leftmost downsteps. Reading off the downsteps in order gives $w^\lambda$. An example is shown in Figure~\ref{fig:read-wlambda}.
\begin{figure}[h!]
    \centering
    \begin{tikzpicture}[scale = 0.5]
    \draw (0,0) -- ++(4,4) --++(3,-3) -- ++(1,1) --++(2,-2) -- cycle;
    \draw (2,2) -- ++(-1,1) --++(1,1) -- ++(-1,1) --++(1,1) -- ++(-1,1) --++(1,1) -- ++(6,-6);
    \node at (3,5) {$\lambda$};
    \begin{scope}[xshift=-0.5em,yshift=-0.6em]
    \node[blue] at (0.5,1) {$5$};
    \node[blue] at (1.5,2) {$4$};
    \node[blue] at (2.5,3) {$\bar1$};
    \node[blue] at (3.5,4) {$\bar2$};
    \node[blue] at (7.5,2) {$\bar3$};
    \end{scope}
    \draw[dashed,blue,<->](0.5,0.5) -- (9.5,0.5);
    \draw[dashed,blue,<->](1.5,1.5) -- (6.5,1.5);
    \draw[dashed,blue,<->](2.5,2.5) -- (5.5,2.5);
    \draw[dashed,blue,<->](3.5,3.5) -- (4.5,3.5);
    \draw[dashed,blue,<->](7.5,1.5) -- (8.5,1.5);
    \begin{scope}[xshift=0.5em,yshift=-0.6em]
    \node[red] at (9.5,1) {$5$};
    \node[red] at (6.5,2) {$4$};
    \node[red] at (5.5,3) {$\bar1$};
    \node[red] at (4.5,4) {$\bar2$};
    \node[red] at (8.5,2) {$\bar3$};
    \end{scope}
    \end{tikzpicture}
    \caption{Reading off $w^\lambda=\bar2\bar14\bar35$ for $\lambda=(6,2,1)$ and $(d,r)=(3,2)$}
    \label{fig:read-wlambda}
\end{figure}
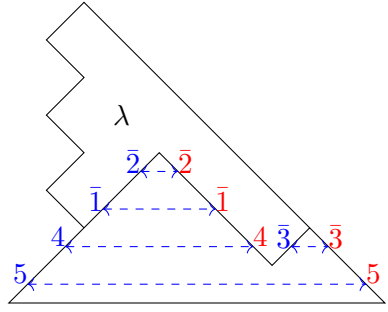
\end{remark}

\subsection{Bijection between \texorpdfstring{$\BS(\lambda)$}{} and \texorpdfstring{$\BS(Z(d,r))|_{\lambda}$}{}}

Recall some notations from the beginning of this section: $\mu_i=Z(d,r)_i-\lambda_i$, $\sigma_i=\sum_{k=1}^i\mu_k$ and $N=|\lambda|$. 
\begin{defin}
    Define $\BS(Z(d,r))|_{\lambda}$ to be the set of balanced tableaux $T$ of shape $Z(d,r)$ such that for all $i\in [d]$ and any $k\in [N+\sigma_{i-1}+1,N+\sigma_i]$, $k$ appears in row $i$ of $T$.
\end{defin}

\begin{lemma}\label{lemma:swap-add}
    Let $B\in \BS(\lambda)$ and fix some $i\in [d]$ such that either $i=1$ or $\lambda_{i-1}\geq\lambda_i+3$. Denote $\lambda^{\#}$ the shifted diagram obtained from $\lambda$ by adding a box in the $i$-th row. Let $j$ be the column index of the box $\lambda^{\#}\setminus\lambda$.
    Let $B^{\#}$ be the tableau obtained from $B$ by 
    \begin{enumerate}
        \item interchange column $j$ and $j+1$ of $B$,
        \item define $B^{\#}(i,j) = N+1$.
    \end{enumerate}
    Then $B^{\#}$ is a balanced tableau and the following map is a bijection:
    \begin{align}\label{eqn:swap&add}
    \begin{split}
        f_i:\BS(\lambda) \longrightarrow &\{T\in \BS(\lambda^{\#}):T(i,j) = N+1\}\\
        B \longmapsto &B^{\#}.
    \end{split}
    \end{align}
    %$B\rightarrow B^+$ is a bijection between $\BS(\lambda)$ and $\{T\in \BS(\lambda^+):T(i,j) = |\lambda|+1\}$.
\end{lemma}
\begin{proof}
    Let $rk_B,\tilde{H}_{B}$ and $rk_{B^{\#}}, \tilde{H}_{B^{\#}}$ be the rank function and extended hooks of $B$ and $B^{\#}$ respectively.
    We first verify that $B^{\#}(a,b)$ is the $rk_{B^{\#}}(a,b)$-th largest entry in $\tilde H_{B^{\#}}(a,b)\subset \tilde {B^{\#}}$ in the following four cases as in Figure~\ref{fig:4cases}:
    
    \emph{Case 1}: $a = i$, $b\leq j$. If $(a,b) = (i,j)$, then $rk_{B^{\#}}(a,b) = 1$. Since there is exactly one element in $\tilde H_{B^{\#}}(a,b)$, we are done. Now suppose $b<j$. Here we have 
    \[rk_{B^{\#}}(a,b) = rk_{B}(a,b)+1\ \text{and } \tilde{H}_{B^{\#}}(a,b) = \tilde{H}_{B}(a,b) \cup \{(i,j)\}.\] Since $B$ is balanced, $B(a,b)$ is the $rk_{B(a,b)}$-th largest element in $\tilde{H}_{B}(a,b)$. $B^{\#}(a,b)$ is then the $rk_{B^{\#}}(a,b)$-th largest entry in $\tilde{H}_{B^{\#}}(a,b)$ since $B(a',b') = B^{\#}(a',b')$ for all $(a',b')\in \tilde{H}_B(a,b)$ and $B^{\#}(i,j) = N+1$ is the largest entry in $\tilde{H}_{B^{\#}}(a,b)$.
    
    \emph{Case 2}: $a<i$, $b=-d-1+i$. Here we again have
    \[rk_{B^{\#}}(a,b) = rk_{B}(a,b)+1\ \text{and } \tilde{H}_{B^{\#}}(a,b) = \tilde{H}_{B}(a,b) \cup \{(i,j)\}.\]
    Since $B(a,j) = B^{\#}(a,j+1)$ and $B(a,j+1) = B^{\#}(a,j)$ and the entries not in column $j$ or $j+1$ remain unchanged under $f_i$, we obtain
    \[\{\tilde{B}^{\#}(a',b'):(a',b')\in \tilde{H}_{B^{\#}}(a,b)\} = \{\tilde{B}(a',b'):(a',b')\in \tilde{H}_{B}(a,b)\}\cup\{N+1\}.\]
    Since $N+1>B^{\#}(a,b)$ and $B$ is balanced, $B^{\#}(a,b)$ is the $rk_{B^{\#}}(a,b)$-th largest entry in $\tilde{H}_{B^{\#}}(a,b)$.
    
    \emph{Case 3}: $a<i$, $b=j$. Since $\lambda_d>0$, we have $j>0$ and thus 
    \begin{equation}\label{eqn:Mar16bbb}
        rk_{B^{\#}}(a,b) = rk_{B}(a,b) = rk_{B}(a,b+1)+1.
    \end{equation}
    Since we swapped column $j$ with $j+1$, we have $B^{\#}(a,b) = B(a,b+1)$ and
    \begin{equation}\label{eqn:Mar16aaa}
        \{B^{\#}(a',b'):(a',b')\in H_{B^{\#}}(a,b)\} = \{B(a',b'):(a',b')\in H_{B}(a,b+1)\}\cup\{B(a,b),N+1\}.
    \end{equation}
    Since $B$ is balanced, $B(a,b+1)$ is the $rk_{B}(a,b+1)$-th largest entry in $H_{B}(a,b+1)$. By Lemma~\ref{lm:same-col-order}, $B(a,b+1)>B(a,b)$. Since $B(a,b+1)<N+1$, by \eqref{eqn:Mar16aaa} and \eqref{eqn:Mar16bbb}, we conclude that $B^{\#}(a,b)$ is the $rk_{B^{\#}}(a,b)$-th largest entry in $H_{B^{\#}}(a,b)$.
    
    \emph{Case 4}: $a<i$, $b=j+1$. Again since $j>0$, 
    \begin{equation}\label{eqn:Mar16ccc}
        rk_{B^{\#}}(a,b) = rk_{B}(a,b) = rk_{B}(a,b-1)-1.
    \end{equation}
    Since we swapped column $j$ and $j+1$, we have $B^{\#}(a,b) = B(a,b-1)$ and
    \begin{equation}\label{eqn:Mar16ddd}
        \{B^{\#}(a',b'):(a',b')\in H_{B^{\#}}(a,b)\} = \{B(a',b'):(a',b')\in H_{B}(a,b-1)\}\setminus \{B(a,b)\}.
    \end{equation}
    Since $B$ is balanced, $B(a,b-1)$ is the $rk_{B}(a,b-1)$-largest entry in $H_{B}(a,b-1)$. By Lemma~\ref{lm:same-col-order}, $B(a,b-1)<B(a,b)$. We can then conclude that $B^{\#}(a,b)$ is the $rk_{B^{\#}}(a,b)$-th largest entry in $H_{B^{\#}}(a,b)$ by \eqref{eqn:Mar16ccc} and \eqref{eqn:Mar16ddd}.
    
    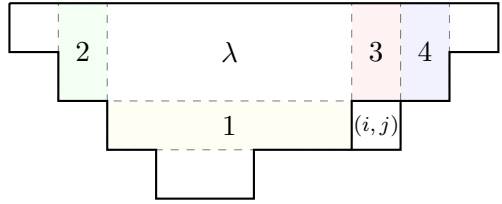
\begin{figure}[h!]
    \begin{center}
    \begin{tikzpicture}[scale = 1.3]
        %\draw[black, thick] (4,-1) -- (4,0);
        %\draw[black, thick] (3.5,-1) -- (3.5,0);
        %\draw[black, thick] (4.5,-0.5) -- (4.5,0);
        \draw (3.75,-1.25) node{\scriptsize{$(i,j)$}};
        \draw (2.25, -0.5) node{$\lambda$};
        \filldraw[color=gray, fill=yellow!5, dashed] (1,-1.5) rectangle (3.5,-1);
        \draw (2.25, -1.25) node{$1$};
        \filldraw[color=gray, fill=green!5, dashed] (0.5,-1) rectangle (1,0);
        \draw (0.75, -0.5) node{$2$};
        \filldraw[color=gray, fill=red!5, dashed] (3.5,-1) rectangle (4,0);
        \draw (3.75, -0.5) node{$3$};
        \filldraw[color=gray, fill=blue!5, dashed] (4,-1) rectangle (4.5,0);
        \draw (4.25, -0.5) node{$4$};
        %\draw[gray, dashed] (1.5,4) -- (1.5,0);
        %\draw[gray, dashed] (2.5,4) -- (2.5,0);
        %\draw[gray, dashed] (3.5,4) -- (3.5,0);
        %\draw[black, thick] (1,0) -- (1,0.5) -- (2,0.5) -- (2,1) -- (3.6,1) -- (3.6,2.4) -- (4,2.4);
        %\draw (1.15,0.75) node{$a_1$};
        %\draw[very thin, gray] (1,0.5)--(1,1);
        %\draw[very thin, gray] (1.25,0.5)--(1.25,1);
        %\draw (3.75, 3.25) node{$a_2$};
        %\draw[very thin, gray] (3.6,2.4)--(3.6,4);
        %\draw[very thin, gray] (3.9,2.4)--(3.9,4);
        %\draw (3,0.5) node{$\lambda^c$};
        %\draw (0.75,4.2) node{$S_4$};
        %\draw (2,4.2) node{$S_3$};
        %\draw (3,4.2) node{$S_2$};
        %\draw (3.75,4.2) node{$S_1$};
        \draw[black, thick] (0,0) -- (5,0) -- (5,-0.5) -- (4.5,-0.5) -- (4.5,-1) -- (3.5,-1) -- (3.5,-1.5) -- (2.5,-1.5) -- (2.5,-2) -- (1.5,-2) -- (1.5,-1.5) -- (1,-1.5) -- (1,-1) -- (0.5,-1) -- (0.5,-0.5) -- (0,-0.5) -- (0,0);
        \draw[black, thick] (4,-1) -- (4,-1.5) -- (3.5,-1.5);
        
    \end{tikzpicture}
\end{center}
    \caption{Location of $(a,b)$ in the four cases of Lemma~\ref{lemma:swap-add}}
    \label{fig:4cases}
\end{figure}
    
    For every $(a,b)\in \lambda^{\#}$ that is not in the four cases above, we have $B^{\#}(a,b) = B(a,b)$ and 
    \[\{B^{\#}(a',b'):(a',b')\in H_{B^{\#}}(a,b)\} = \{B(a',b'):(a',b')\in H_{B}(a,b)\}.\]
    Therefore the balanced condition of $B^{\#}$ follows from that of $B$ and thus $B^{\#}$ is balanced.
    
    Now we are left to show that $f_i$ is a bijection. For a balanced tableau $B$ of shape $\lambda^{\#}$ such that $B(i,j) = N+1$, define $B^{\flat}$ to be the tableau obtained from $B$ by 
    \begin{enumerate}
        \item remove the entry $B(i,j)$,
        \item interchange column $j$ and $j+1$.
    \end{enumerate}
    The fact that $B^{\flat}$ is a balanced tableau follows from the same reasoning as the analysis above.
    It is then clear that the map
    \begin{align}
    \begin{split}
        g_i:\{T\in \BS(\lambda^{\#}):T(i,j) = N+1\} \longrightarrow &\BS(\lambda)\\
        B \longmapsto &B^{\flat}
    \end{split}
    \end{align}
    is the inverse of $f_i$ and thus $f_i$ is a bijection.
\end{proof}

\begin{lemma}\label{lm:BStoBSlambda}
    $\BS(Z(d,r))|_{\lambda}$ is the image of $\BS(\lambda)$ under the composition of maps $F= (f_d)^{a_d}\circ (f_{d-1})^{a_{d-1}}\circ\cdots\circ (f_1)^{a_1}$ with each $f_i$ defined as in \eqref{eqn:swap&add}. As a result, $F$ is a bijection between $\BS(\lambda)$ and $\BS(Z(d,r))|_{\lambda}$.
\end{lemma}
\begin{proof}
    Notice by the construction in Lemma~\ref{lemma:swap-add}, for any tableau $T$ and any $(a,b)\in \sh(T)$, the entry $T(a,b)$ appear in row $a$ of $f_i(T)$ for all $i\in [d]$. Now by definition $N+b_i+j$ appear in row $i$ of $(f_i)^j\circ (f_{i-1})^{a_{i-1}}\circ\cdots \circ (f_1)^{a_1}(B)$ for all $i\in [d],1\leq j\leq a_i$, and all $B\in\BS(\lambda)$. Thus for all $i\in [d]$, $k\in [N+b_i+1,N+b_{i+1}]$ and all $B\in \BS(\lambda)$, $k$ appears in row $i$ of $F(B)$. Therefore $F(B)\in \BS(Z(d,r))|_{\lambda}$. 
    
    Similarly, notice that for any tableau $T$ and any $(a,b)\in \sh(T)$ such that $T(a,b)$ is not the largest entry, $T(a,b)$ appears in row $a$ in $(f_i)^{-1}T$. As a result, for any $k\leq N$ and any $T\in \BS(Z(d,r))|_{\lambda}$, $k$ appears in the same row of $T$ as $F^{-1}(T)$. Therefore $F^{-1}(T)\in \BS(\lambda)$ and we can conclude that $\BS(Z(d,r))|_{\lambda}$ is the image of $\BS(\lambda)$ under $F$.  
\end{proof}

\subsection{Bijection between \texorpdfstring{$\BS(Z(d,r))|_{\lambda}$}{} and \texorpdfstring{$\Red(w^{\lambda})$}{} }

\begin{prop}\label{lm:BSlambdatoRed}
    Let $\mathbf{a}\in \Red(w^{(d,r)})$ be a reduced word. Then $\ro(\mathbf{a})$ gives a balanced tableau in $\BS(Z(d,r))|_{\lambda}$ if and only if the ending segment of $\mathbf{a}$ is the same as $\mathbf{a^\lambda}$ as in Proposition~\ref{prop:alambda-and-Plambda}. Consequently, this induces a bijection between $\BS(Z(d,r))|_{\lambda}$ and $\Red(w^\lambda)$.
\end{prop}
\begin{proof}
    Set $\mathbf{a}=a_1a_2\cdots a_\ell\in\Red(w^{(d,r)})$, and $\ro(\mathbf{a})=\gamma_1,\gamma_2,\dots,\gamma_\ell$. By construction (Lemma~\ref{lemma:swap-add}), $\ro(\mathbf{a})$ produces a balanced tableau in $\BS(Z(d,r))|_{\lambda}$ if and only if for all $i\in [d]$, the roots $\gamma_{N+\sigma_{i-1}+1},\dots,\gamma_{N+\sigma_i}$ appear in the $i$-th row of the labeling of $Z(d,r)$ as in \eqref{eqn:defrootlabel}.
    Now we use induction on $k$ to prove the following two conditions on $\mathbf{a}$ are equivalent. Notice that the proposition is precisely the case $k=d$.
    \begin{enumerate}
        \item $\mathbf{a}$ ends with $a_{N+\sigma_{d-k}+1}\cdots a_\ell=\mathbf{a_{d-k+1}^\lambda}\mathbf{a_{d-k+2}^\lambda}\cdots\mathbf{a_d^\lambda}$;
        \item The roots $\gamma_{N+\sigma_{i-1}+1},\dots,\gamma_{N+\sigma_i}$ appear in the $i$-th row of the labeling of $Z(d,r)$ for all $d-k+1\leq i\leq d$.
    \end{enumerate}
    When $k=0$, both conditions are vacuous. Assume the two conditions are equivalent for $k$. Consider the permutation $w'=w^{(d,r)}s_{a_\ell}s_{a_{\ell-1}}\cdots s_{a_{N+\sigma_{d-k}+1}}$. By inductive hypothesis we know that $a_{N+\sigma_{d-k}+1}\cdots a_\ell=\mathbf{a_{d-k+1}^\lambda}\mathbf{a_{d-k+2}^\lambda}\cdots\mathbf{a_d^\lambda}$. By \eqref{eqn:Mar21aaa}, the subsequence $\mathbf{a_{d-k+1}^\lambda}\mathbf{a_{d-k+2}^\lambda}\cdots\mathbf{a_d^\lambda}$ only contains indices in the interval $[1,r+k-1]$. Therefore, $w'(t) = w^{(d,r)}(t)$ for all $t\geq d+k+1$ since they are not affected by the simple transpositions. Thus the one-line notation of $w'$ ends with $\overline{k+1}\,\overline{k+2}\cdots\overline{d}$. Now consider the following steps:
    \begin{equation}\label{eqn:Mar21bbb}
    \begin{tikzcd}[column sep=6em]
        w' & w''\arrow{l}{a_{N+\sigma_{d-k}}}[swap]{\gamma_{N+\sigma_{d-k}}} &
        \cdots\arrow{l}{a_{N+\sigma_{d-k}-1}}[swap]{\gamma_{N+\sigma_{d-k}-1}} &
        w'''\arrow{l}{a_{N+\sigma_{d-k-1}+1}}[swap]{\gamma_{N+\sigma_{d-k-1}+1}}.
    \end{tikzcd}
    \end{equation}
    To show (2)$\implies$(1), we assume the $\gamma$'s in \eqref{eqn:Mar21bbb} all appear in the $(d-k)^\text{th}$ row of $Z(d,r)$, which contains $e_{k+1}+e_{q}$ for $q>d$ or $e_{k+1}-e_i$ for $i<k+1$. All these roots include $e_{k+1}$, so these steps must swap $\overline{k+1}$ forward in $w'$. As a result, $a_{N+\sigma_{d-k-1}+1}\cdots a_{N+\sigma_{d-k}}=\mathbf{a_{d-k}^\lambda}$. To show (1)$\implies$(2), if these steps are exactly $\mathbf{a_{d-k}^\lambda}$, since $w'(r+k+1)=\overline{k+1}$, these steps must swap $\overline{k+1}$ forward in $w'$, so the roots $\gamma$'s all lie in the $(d-k)^\text{th}$ row of $Z(d,r)$. This concludes the induction step and we are done by induction.
\end{proof}

\begin{proof}[Proof of Theorem~\ref{thm:main}]
Combining Corollary~\ref{cor:RedtoSYTlambda}, Lemma~\ref{lm:BStoBSlambda} and Proposition~\ref{lm:BSlambdatoRed}, we get a bijection between $\SYT(\lambda)$ and $\BS(\lambda)$.
\end{proof}

\begin{comment}
\begin{theorem}\label{thm:main}
There is a bijection between $\BS(\lambda)$ and $\SYT(\lambda)$ for any shifted shape $\lambda$.
\end{theorem}
\begin{proof}
By combining Corollary~\ref{cor:RedtoSYTlambda}, Lemma~\ref{lm:BStoBSlambda} and Lemma~\ref{lm:BSlambdatoRed}, we get a bijection
\[
\begin{tikzcd}
\BS(\lambda)\arrow[r,"\sim"] & \BS(Z(d,r))|_{\lambda}\arrow[r,"\sim"] & \Red(w^{\lambda})\arrow[r,"\sim"] &
\SYT(Z(d,r))|_{\lambda}\arrow[r,"\sim"] &
\SYT(\lambda).
\end{tikzcd}
\]
\end{proof}
\end{comment}

\begin{ex}\label{ex:main-ex}
We now work out an example in the case where $\lambda=(6,2,1)$, $d=3$ and $r=2$. Assume we start with a balanced tableaux $B\in\BS(\lambda)$ shown here: \[\ytableausetup{boxsize=1.4em}
B = \begin{ytableau}
6 & 3 & 4 & 1 & 5 & 9 \\
\none & 7 & 8 & \none & \none \\
\none & \none & 2 & \none 
\end{ytableau}\, .\]
We can complete it to $B^+\in \BS(Z(3,2))$ using the algorithm in Lemma~\ref{lm:BStoBSlambda}
\begin{center}
\ytableausetup{boxsize=1.5em}
$B = $
\begin{ytableau}
    6 & 3 & 4 & 1 & 5 & 9 \\
    \none & 7 & 8 & \none & \none \\
    \none & \none & 2 & \none
\end{ytableau}
$\rightarrow$ 
\begin{ytableau}
    6 & 3 & 4 & 1 & 5 & 9 & \circled{10} \\
    \none & 7 & 8 & \none & \none \\
    \none & \none & 2 & \none
\end{ytableau}
$\rightarrow$
\begin{ytableau}
    6 & 3 & 4 & 5 & 1 & 9 & \circled{10} \\
    \none & 7 & 8 & \circled{11} & \none \\
    \none & \none & 2 & \none
\end{ytableau}
\vspace{0.2cm}
$\rightarrow$
\begin{ytableau}
    6 & 3 & 4 & 5 & 9 & 1 & \circled{10} \\
    \none & 7 & 8 & \circled{11} & \circled{12} \\
    \none & \none & 2 & \none
\end{ytableau}
$\rightarrow$
\begin{ytableau}
    6 & 3 & 4 & 5 & 9 & \circled{10} & 1 \\
    \none & 7 & 8 & \circled{11} & \circled{12} & \circled{13}\\
    \none & \none & 2 & \none
\end{ytableau}
$\rightarrow$
\begin{ytableau}
    6 & 3 & 4 & 9 & 5 & \circled{10} & 1 \\
    \none & 7 & 8 & \circled{12} & \circled{11} & \circled{13}\\
    \none & \none & 2 & \circled{14}
\end{ytableau}
\vspace{0.2cm}
$\rightarrow$
\begin{ytableau}
    6 & 3 & 4 & 9 & \circled{10} & 5 & 1 \\
    \none & 7 & 8 & \circled{12} & \circled{13} & \circled{11}\\
    \none & \none & 2 & \circled{14} & \circled{15}
\end{ytableau}
$= B^+$
\end{center}
Now $B^+$ gives a reflection order of $w^{(d,r)}$ as follows
\[
\begin{tikzcd}[row sep = small]
12345 \arrow[r, "e_3-e_2"]& 13245 \arrow[r, "e_1"] & \bar13245 \arrow[r, "e_3+e_1"] & 3\bar1245 \arrow[r, "e_3"] & \bar3\bar1245 \arrow[r, "e_3-e_1"] & \bar1\bar3245\\
\arrow[r, "e_3+e_2"]& \bar12\bar345 \arrow[r, "e_2+e_1"] & 2\bar1\bar345 \arrow[r, "e_2"] & \bar2\bar1\bar345 \arrow[r, "e_4+e_3"] & \bar2\bar14\bar35 \arrow[r, "e_5+e_3"] & \bar2\bar145\bar3\\
\arrow[r, "e_2-e_1"]& \bar1\bar245\bar3 \arrow[r, "e_4+e_2"] & \bar14\bar25\bar3 \arrow[r, "e_5+e_2"] & \bar145\bar2\bar3 \arrow[r, "e_4+e_1"] & 4\bar15\bar2\bar3 \arrow[r, "e_5+e_1"] & 45\bar1\bar2\bar3.\\
\end{tikzcd}
\]
We can read off the reduced word $\mathbf{a}=201012103412312\in\Red(w^{(3,2)})$. We can confirm that the reduced word ends with $\mathbf{a^\lambda}=412312$. Now we perform the Kra\'skiewicz's insertion on $\mathbf{a}$ described in Section~\ref{sec:kraskiewicz} and we get
\[\ytableausetup{boxsize=1.4em}
P(\mathbf{a}) = \begin{ytableau}
4 & 3 & 0 & 1 & 2 & 3 & 4\\
\none & 3 & 0 & 1 & 2 & 3\\
\none & \none &0 & 1 & 2
\end{ytableau},\ 
Q(\mathbf{a}) = \begin{ytableau}
1 & 2 & 3 & 5 & 6 & 9 & 10\\
\none & 4 & 7 & 11 & 12 & 13\\
\none & \none & 8 & 14 & 15
\end{ytableau},
\]
Finally, let $T^+=Q(\mathbf{a})\in\SYT(w^{(3,2)})$, and $T\in\SYT(\lambda)$ is obtained from $T^+$ by deleting the largest entries until $|\lambda|$ entries are left:
\[\ytableausetup{boxsize=1.5em}
T^+ = \begin{ytableau}
1 & 2 & 3 & 5 & 6 & 9 & \circled{10}\\
\none & 4 & 7 & \circled{11} & \circled{12} & \circled{13}\\
\none & \none & 8 & \circled{14} & \circled{15}
\end{ytableau},\ 
T = \begin{ytableau}
1 & 2 & 3 & 5 & 6 & 9\\
\none & 4 & 7\\
\none & \none & 8
\end{ytableau}.
\]
\end{ex}

% \begin{lemma}
%     Consider the labeling of $Z(d,r)$ as in \eqref{eqn:defrootlabel}. Then $\Inv(w^{\lambda})$ contains $\lambda_i$ roots from the $i$-th row of $Z(d,r)$ for all $i\in [d]$.
% \end{lemma}
% \begin{proof}

% \end{proof}
% TODO: Let $\mathbf{a}\in \Red(w^{d,r})$ be a reduced word such that the reflection order of $\mathbf{a}$ gives a balanced tableau in $\BS(Z(d,r))|_{\lambda}$, want to show that the ending segment is the same as $\mathbf{a'}$ as in Section 5.1.

%% TO DO: 
%% 1. Define \SYT(Z(d,r))|_{\lambda} and \BS(Z(d,r))|_{\lambda},namely SYT or BS such that |\lambda|+1,...,|\lambda|+a_1 appears in the first row, |\lambda|+a_1+1,...,|\lambda|+a_1+a_2 in the second row etc, where a_i = number of boxes in Z(d,r)/\lambda. (Done)
%% 2. Define $w^{\lambda}$ and show that the insertion gives a bijection between \Red(w^{\lambda}) and \SYT(Z(d,r))|_{\lambda} (Done)
%% 3. Prove the "swap column and add n" gives a bijection and define \BS(Z(d,r))|_{\lambda}
%% 4. Show that the left inversion set of w^{\lambda} has \lambda_1 boxes in the first row of Z(d,r), \lambda_2 boxes in the second etc. This proves that we have an injection from \Ref(w^{\lambda}) to BS(Z(d,r))|_{\lambda} via reflection order.
%% 5. Prove a bijection between \Red(w^\lambda) and \BS(Z(d,r))|_{\lambda} by characterizing the reflection order of (w^{\lambda})^-1 w^{d,r} (after w^\lambda)

\section*{Acknowledgements}
We are grateful to Alex Yong for helpful conversation. SG was partially supported by NSF RTG grant DMS 1937241 and an NSF Graduate Research Fellowship under grant No. DGE-1746047.

\bibliographystyle{plain}
\bibliography{ref.bib}
\end{document}